\documentclass[12pt,a4paper]{amsart}

\usepackage{amssymb,amscd,latexsym,amsxtra,dsfont,enumerate}

\usepackage[markup=underlined]{changes}
\definechangesauthor[color=red]{Curt}
\definechangesauthor[color=blue]{Emanuel}
\usepackage{amsfonts,bbold}
\usepackage[mathscr]{eucal}
\usepackage{calrsfs}
\usepackage{enumitem}
\usepackage{fge}
\usepackage{mathtools}
\usepackage{scalerel}
\usepackage[autosize]{dot2texi}

\definecolor{britishracinggreen}{rgb}{0.0, 0.26, 0.15}
\definecolor{amaranth}{rgb}{0.9, 0.17, 0.31}

\newtheorem{theorem}{Theorem}
\newtheorem{corollary}[theorem]{Corollary}
\newtheorem{lemma}[theorem]{Lemma}
\newtheorem{proposition}[theorem]{Proposition}

\theoremstyle{definition}
\newtheorem{definition}[theorem]{Definition}

\newcommand{\mybh}{$\mathscr{B}(\mathscr{H})$}

\newcommand\restr[2]{{
		\left.\kern-\nulldelimiterspace 
		#1 
		\vphantom{\big|} 
		\right|_{#2} 
}}

\newcommand{\mynorm}[1]{ \left\| #1 \right\| }
\newcommand{\myinnerproduct}[2]{ \langle #1, #2 \rangle }
\newcommand{\myabs}[1]{ | #1 | }
\newcommand{\mybhplusplus}{$\mathscr{B}(\mathscr{H})^{++}$}

\begin{document}
	\title{On a comprehensive review of a proof of L\"owner's theorem}
	\author{Curt Healey}
	\address{Curt Healey\\
		Department of Mathematics\\
		Faculty of Science\\
		University of Malta\\
		Msida MSD 2080  Malta}
	\email{curt.c.healey.13@um.edu.mt}

	\date{\today}
	\begin{abstract}
		Recent studies in Kubo-Ando theory make frequent use of the relationship between Kubo-Ando connections and positive operator monotone functions. This relationship is deeply connected to L\"owner's theorem and our aim is to provide a comprehensive review of one of the proofs of L\"owner's theorem. Our motivation arises from the fact that the foundational components upon which the theorem rests are found within a variety of sources, rendering it difficult to obtain a complete understanding of the proof without engaging in substantial external consultation. By consolidating these elements into a single, continuous account, the proof becomes substantially more accessible and may be assimilated with greater clarity and efficiency.
	\end{abstract}
	\maketitle
	
	\section{Motivation}
	Recently, various problems regarding Kubo-Ando connections have been solved which were stated in \cite[Page 22]{molnar-main}.  We recall that a binary operation $\sigma$ on the positive definite cone \mybhplusplus\,of the algebra \mybh\, of bounded operators on the Hilbert space $\mathscr{H}$,  is called a \emph{Kubo-Ando connection} if it satisfies the following properties:
	\begin{enumerate}[label=\textnormal{(\Roman*)}]
		\item If $A \leq C$ and $B \leq D$, then $ A \sigma B \leq C \sigma D $.
		\item $ C (A \sigma B) C \leq (CAC) \sigma (CBC) $.
		\item  If $ A_n  \downarrow A$ and $ B_n \downarrow B $, then $ A_n \sigma B_n \downarrow A \sigma B$ (We write $A_n\downarrow A$ when $(A_n)$ is monotonic decreasing and SOT-convergent to $A$).
	\end{enumerate}
	Furthermore, 	 In \cite[Theorem  3.2]{kubo-ando-main}, it is shown that there is an affine order isomorphism from the class of Kubo-Ando connections onto the class of operator monotone functions via the map $f(xI) = I \sigma (xI)$ for $x>0$. 	Since every finite Borel measure on $[0, \infty]$ is regular (i.e a Radon measure), we have that operator monotone functions correspond to positive finite Borel measures on $[0,\infty]$ by L\"owner's Theorem (see \cite{donoghue}): To every  operator monotone function $f$ corresponds a unique positive and finite Borel measure $m$ on $[0,\infty]$ such that
	\begin{equation}\label{e4}
		f(x)=\,\int_{[0,\infty]}\frac{x(1+t)}{x+t}\,{\rm d}m(t)\,=\,m(\{0\})\,+\,x\,m(\{\infty\})\,+\,\int_{(0,\infty)}\frac{1+t}{t}(t:x)\,{\rm d}m(t)\quad(x>0).
	\end{equation}
	It is easy to see that $f(0+)=m(\{0\})$, $f^\circ(0+)=m(\{\infty\})$.
	Finally, by \cite[Theorem 3.4]{kubo-ando-main}, there is an affine isomorphism from the class of positive finite Borel measures on $[0,\infty]$ onto the class of Kubo-Ando connections. This is given by the formula
	\begin{equation}\label{kubo-ando-integral-representation}
		A \sigma B = \alpha A +  \beta B + \int_{(0, \infty)} \frac{1+t}{t} (tA:B) \, {\rm d} m(t) \quad A, B \in  \mathscr B(\mathscr{H})^{+}
	\end{equation}
	where $\alpha = m(\lbrace0\rbrace)$ and $\beta =m(\lbrace \infty \rbrace)$.
	In the case of a symmetric Kubo-Ando connection, $ \alpha = \beta $. For further details on the provenance of the integral representation (\ref{kubo-ando-integral-representation}), the reader is referred to \cite[Theorem 3.2]{kubo-ando-main}. 
	\par 
	The connection is clear and we give here a comprehensive review of the proof for L\"owner's theorem which is found in \cite{hansen-fasttrack}. The original exposition assumes a working knowledge of both Choquet theory and matrix analysis, which may pose a challenge to readers unfamiliar with these areas. To improve accessibility and facilitate a deeper understanding of the proof, we have chosen to include not only the theorem itself but also the necessary foundational concepts. By consolidating the background material and the proof within a single, self-contained narrative, we aim to streamline the learning process and reduce the need for external references. This integrated approach allows the reader to follow the logical structure of the argument more easily, appreciate the interplay between different mathematical tools, and develop a more intuitive grasp of the theorem's significance and underlying principles.

	\section{Choquet Theory}
	\subsection{Introduction}
	In this section we develop Choquet theory for metrizable spaces. Let us first provide some motivation for the theory by stating some results regarding convex compact subsets in $\mathbb{R}^{n}$. We give here Carath\'eodory's and Minkowski's Theorems which can be found in \cite[Theorem 3.1.2]{niculescu} and \cite[Theorem 3.3.5]{niculescu} respectively. 
	\begin{theorem}[Carath\'eodory]
		Suppose that $E$ is a subset of a linear space $X$ and its convex hull $\textnormal{co}(E)$ has dimension $n$. Then each point $x \in \textnormal{co}(E)$ is the convex combination of at most $n+1$ points of $E$.
	\end{theorem}
	
	\begin{theorem}[Minkowski]
		Every nonempty convex and compact subset of $\mathbb{R}^{n}$ is the convex hull of its extreme points.
	\end{theorem}
	
	Therefore, by combining the above two theorems, we can arrive at the conclusion that every point of a convex compact subset  $K \subseteq \mathbb{R}^{n}$ can be written as a convex combination of its extreme points. Let $\textnormal{Ext}(K) = \lbrace e_i \rbrace_{i=1}^{n+1}$ denote the extreme points of $K$. If $x \in K $, then $ x = \sum_{i=1}^{n+1} e_i \lambda_i $ where $\sum_{i=1}^{n+1} \lambda_i = 1$ and $\lambda_{i} \geq 0$ for each $i$. Suppose  that to each extreme point $e_{i}$ we associate the nonnegative regular Borel measure $\mu_{e_i}$ defined by
	\[
	\mu_{e_i}(S) = 
	\begin{cases} 
		1 & \text{if } e_i \in S \\
		0 & \text{if } e_i \notin S
	\end{cases},
	\]
	where $S \subseteq K$ is a Borel set. Let $\mu = \sum_{i=1}^{n+1} \mu_{e_i} \lambda_{i} $ which is also a nonnegative regular Borel measure such that $\mu(K) = 1$. Since $\mu(\lbrace e_i \rbrace) = \lambda_i$ for any $1 \leq i \leq n+1$, we have, $ x= \sum_{i=1}^{n+1} \mu(\lbrace e_i \rbrace) e_i$, and if $\rho$ is a continuous linear functional on $K$, then $\rho(x) = \sum_{i=1}^{n+1} \mu(\lbrace e_i \rbrace) \rho(e_i) = \int_{Ext(K)}  \rho\,  \rm{d} \mu $. Let us now define what we mean when we say that $\mu$ \emph{represents} $x$ \cite[Pg 1 Definition]{phelps}.
	
	\begin{definition}
		Suppose that $X$ is a nonempty compact subset of a locally convex space $E$, and that $\mu$ is a probability measure on $X$. $\textnormal{(}$That is $\mu$ is a regular Borel measure on $X$ such that $\mu(X) = 1$$\textnormal{)}$. A point $x \in E$ is said to be represented by $\mu$ if $\rho(x) = \int_{X} \rho  \,  {\rm d} \mu$ for every continuous real-valued linear functional $\rho$ on $E$.
	\end{definition}
	
	Let us now define what we mean when we say that a nonnegative regular Borel measure $\mu$ is supported by a Borel set $S$.
	
	\begin{definition}
		If $\mu$ is a nonnegative regular Borel measure on the compact Hausdorff sapce $X$ and $S$ is a Borel subset of $X$, then we say that $\mu$ is supported by $S$ if $\mu(X \setminus S) = 0$.
	\end{definition}
	
	This leads us to the theorem of interest \cite[Page 14, Theorem (Choquet)]{phelps}.
	
	\begin{theorem}[Choquet]
		Suppose that $X$ is a metrizable compact convex subset of a locally convex space E, and that $x_0$ is an element of $X$. Then there is a probability measure $\mu$ on $X$ which represents $x_0$ and is supported by the extreme points of $X$.
	\end{theorem}
	
	The proof of the theorem requires certain preliminary results on the space $C(X)$ of complex valued functions on some arbitrary topological space $X$.
	
	\subsection{Preliminary results}
	Let us first recall a sufficient condition for a continuous function to become uniformly continuous between metric spaces.
	\begin{proposition}\label{uniform-continuity-of-f}
		Let $X, Y$ be metric spaces such that $X$ is compact. Suppose $f: X \rightarrow Y$ is a continuous function, then $f$ is uniformly continuous. 
	\end{proposition}
	
	Let us also recall that if $X$ is a metrizable compact space, then $C(X)$ is separable  \cite[Page 140 Theorem 6.6]{conway}.
	
	\begin{theorem}\label{compact-metric-space-implies-c(x)-separable}
		If $X$ is a compact metric space, then $C(X)$ is separable. 
	\end{theorem}
	
	Let us next recall the definitions for affine, convex, and concave functions.
	\begin{definition}
		\index{Function!affine} \index{Function!convex} \index{Function!concave} Let $V$ be a vector space and $C \subseteq V$ a convex set such that $\lambda \in (0, 1)$ and $x, y \in C$. Then $f: C \rightarrow \mathbb{R}$ is said to be affine if $f(\lambda x + (1-\lambda)y) = \lambda f(x) + (1- \lambda) f(y)$. Furthermore, $f$ is said to be convex if
		$f(\lambda x + (1-\lambda)y) \leq \lambda f(x) + (1- \lambda) f(y)$, and $f$ is said to concave if $-f$ is convex.
	\end{definition}
	
	Consider the set of affine functions $\mathcal{A} \subseteq C(X)$ where $X$ is a convex subset of a locally convex space $E$. Naturally, $\mathcal{A}$ is a subspace of the Banach space $C(X)$. Furthermore, since $\mathcal{A}$ contains all functions of the form $x \rightarrow \rho(x)+ r$ on $X$ where $\rho$ is a continuous real-valued linear functional on $X$ and $r \in \mathbb{R}$, then $\mathcal{A}$ separates $X$ \cite[Corollary 1.2.12]{kadison-vol-1}. Let us first define what we mean by the \emph{upper envelope} of $f$.
	\begin{definition}
		\index{Upper envelope} Let $X$ be a convex subset of a locally convex space $E$ and $f: X \rightarrow \mathbb{R}$ be a bounded function. Let $\mathcal{A}$ denote the set of continuous affine functions on $X$. The upper envelope of $f$ is $\bar{f}(x) = \textnormal{inf} \lbrace h(x) : h \in \mathcal{A} \; \textnormal{and} \; h \geq f \rbrace$ where $x \in X$.
	\end{definition}
	
	Since constant functions are elements of $\mathcal{A}$ and the upper envelope is defined on bounded functions, then the upper envelope is well-defined. Let us go through some properties of the upper envelope \cite[Page 13]{phelps}.
	\begin{proposition}\label{upper-envelope-properties}
		Let $X$ be a subset of a locally convex space $E$ and $f, g : X \rightarrow \mathbb{R}$ be bounded functions.  Let $\mathcal{A}$ denote the set of continuous affine functions on $X$. Then, the upper envelope satisfies the following properties.
		\begin{enumerate}[label=\textnormal{(\Roman*)}]
			\item $\overline{f}$ is bounded, concave, upper semicontinuous \textnormal{(}hence Borel measurable\textnormal{)}, and $f \leq \overline{f}$.
			\item $\overline{f + g} \leq \overline{f} + \overline{g}$, and $\alpha \geq 0$ implies $\overline{\alpha f} = \alpha \overline{f}$.
			\item If $g \in \mathcal{A}$, then $\overline{f + g} = \overline{f} + g$.
			\item If $f$ is concave and upper semicontinuous, then $f=\overline{f}.$
		\end{enumerate}
	\end{proposition}
	
	\begin{proof}
		\emph{\textnormal{(I)}}. We shall only prove that $\overline{f}$ is upper semicontinuous. Consider $x \in \overline{f}^{-1}(-\infty, r)$ where $r \in \mathbb{R}$. There exists $g \in \mathcal{A}$ such that $g \geq f$ and $\overline{f}(x)  < g(x) < r$. Therefore, $x \in g^{-1}(-\infty, r)$ and $ g^{-1}(-\infty, r) \subseteq \overline{f}^{-1}(-\infty, r)$. By continuity of $g$, there is an open set $O \subseteq \mathbb{R}$ such that $x \in O \subseteq \overline{f}^{-1}(-\infty, r)$. Therefore, $\overline{f}^{-1}(-\infty, r)$ is open for any $r \in \mathbb{R}$ which implies that $\overline{f}$ is upper semicontinuous.
		\par
		\emph{\textnormal{(II)}}. Let $x \in X$. If $f(x) \leq h_1(x)$ and $g(x) \leq h_2(x)$, then $ f(x) + g(x) \leq h_1(x) + h_2(x)$. Therefore, $\lbrace h(x): h \geq f, h \in \mathcal{A} \rbrace + \lbrace h(x): h \geq g, h \in \mathcal{A} \rbrace$ $\subseteq \lbrace h(x) : h \geq f + g, \, h \in \mathcal{A} \rbrace$. By taking the infimum, we can conclude that $\overline{f+g}(x) \leq \overline{f}(x) + \overline{g}(x)$ for all $x \in X$. This proves the first part of the result. 
		\par 
		Let $\alpha \geq 0$. Then $\lbrace h(x) : h \geq \alpha f, \, h \in \mathcal{A} \rbrace = \lbrace \alpha h(x): h \geq f, \, h \in \mathcal{A}  \rbrace$. By taking the infimum, we can conclude that $ \overline{\alpha f}(x) = \alpha \overline{f}(x)$ for all $x \in X$.  
		\par 
		\emph{\textnormal{(III)}}. Since $g \in \mathcal{A}$, we have $\overline{g} = g$. By (II), $\overline{f + g} \leq \overline{f} + g$. Moreover, $f + g \leq \overline{f+g}$ implies that $\overline{f} \leq \overline{f + g} - \overline{g}$. Therefore, $\overline{f} + g\leq \overline{f + g}$ which allows us to conclude that $\overline{f + g} = \overline{f} + g$.
		\par 
		\emph{\textnormal{(IV)}} Consider $K \subseteq X \times \mathbb{R}$ such that $K = \lbrace \lbrace x, r \rbrace : f(x) \geq r \rbrace$. Let $x_1, x_2 \in X$, $\lambda \in (0,1)$, and $r_1, r_2 \in \mathbb{R}$. Since $f( \lambda x_1 + (1- \lambda) x_2) \geq \lambda f(x_1) + (1-\lambda) f(x_2) \geq \lambda r_1 + (1-\lambda) r_2$, we have that $K$ is convex. Since the hypograph of an upper semicontinuous function is closed, it follows that $K$ is closed. Suppose that $f(x_0) < \overline{f}(x_0)$ for some $x_0 \in X$. By \cite[Corollary 1.2.12]{kadison-vol-1}, we can find a continuous linear functional $\rho$ such that $\text{sup} \, \rho(K) < \lambda < \rho (\lbrace x_0, \overline{f}(x_0) \rbrace)$ for some $\lambda \in \mathbb{R}$. Since $ \rho( \lbrace x_0, f(x_0) \rbrace) < \rho( \lbrace x_0, \overline{f}(x_0) \rbrace)$ implies that $(\overline{f}(x_0) - f(x_0)) \rho (\lbrace 0,1 \rbrace) > 0$, it follows that $\rho (\lbrace 0,1 \rbrace) > 0$. Let us define a continuous affine function $h : X \rightarrow \mathbb{R}$ such that 
		\[ h(x) = \frac{\lambda - \rho( \lbrace x, 0 \rbrace)}{\rho (\lbrace0,1\rbrace)}, \quad \forall \, x \in X. \]
		Rearranging the terms allows us to conclude that $\rho (\lbrace x, h(x) \rbrace) = \lambda$. Furthermore, since $\rho (\lbrace 0, 1 \rbrace) > 0$, we have
		\begin{equation}
			\begin{split}
				\label{upper-envelope-properties-eq1}
				\rho( \lbrace x, f(x) \rbrace) < \lambda = \rho(\lbrace x, h(x) \rbrace) &\implies 0 < (h(x)-f(x))\rho(\lbrace 0, 1 \rbrace), \quad \forall \, x \in X \\
				&\implies 0 < h(x)-f(x), \quad \forall \, x \in X   \\
				&\implies \overline{f}(x) \leq h(x), \quad \forall \, x \in X.
			\end{split}
		\end{equation}
		Similarly, 
		\[   \rho(\lbrace x_0, h(x_0) \rbrace) = \lambda <  \rho(\lbrace x_0, \overline{f}(x_0) \rbrace) \implies  h(x_0) < \overline{f}(x_0) . \]
		However, this contradicts  (\ref{upper-envelope-properties-eq1}), so $\overline{f} = f$. 
	\end{proof}
	
	\subsection{Proof}
	We follow the proof found in \cite[Page 14, Theorem (Choquet)]{phelps}.
	\begin{lemma}\label{existence-strictly-convex-function}
		Suppose that $X$ is a metrizable compact convex subset of a locally convex space E. Then there exists a strictly continuous convex function $f: X \rightarrow \mathbb{R}$.
	\end{lemma}
	\begin{proof}
		Let $\mathcal{A} \subseteq C(X)$ be the set of affine functions. Since $X$ is metrizable and compact, then $C(X)$ is separable by Theorem \ref{compact-metric-space-implies-c(x)-separable}. Consider
		\[ \mathcal{A}_1 = \lbrace f \in C(X, \mathbb{R}) :f \in \mathcal{A}, \mynorm{f} = 1 \rbrace.\] 
		Since $\mathcal{A}_1$ is separable, there exists a countable dense subset $D_1 = \lbrace f_n \rbrace_{n \in \mathbb{N}}$ of $\mathcal{A}_1$ . We shall construct a strictly convex function from $D_1$. Let $x_1, x_2 \in X$ such that $x_1 \neq x_2$. By \cite[Corollary 1.2.11]{kadison-vol-1}, there exists a continuous real-valued linear functional $\rho$ on $E$ with $\mynorm{\rho} =1 $ s.t $\myabs{\rho(x_1) - \rho(x_2)} > \epsilon$ for some $\epsilon > 0$. Since $\rho_{\vert X} \in \mathcal{A}_1$, there is some $f_i \in D_1$ for some $i \in \mathbb{N}$ such that $\mynorm{\rho_{\vert X} - f_i} \leq \frac{\epsilon}{3}$. Therefore, 
		\begin{equation*}
			\begin{split}
				\myabs{\rho_{\vert X}(x_1) - \rho_{\vert X}(x_2)} &\leq \myabs{\rho_{\vert X}(x_1) - f_i(x_1)} + \myabs{f_i(x_1) - f_i(x_2)} + \myabs{f_i(x_2) - \rho_{\vert X}(x_2)} \\
				&\leq \frac{2 \epsilon}{3} + \myabs{f_i(x_1) - f_i(x_2)}.
			\end{split}
		\end{equation*}
		Thus,
		\begin{equation}\label{existence-strictly-convex-function-eq1}
			\myabs{f_i(x_1) - f_i(x_2)} > \epsilon /3.
		\end{equation}
		
		Let $f = \sum_{n=1}^{\infty}  2^{-n} f_n^{2}$ and let us now show that this series is Cauchy. If $N_1, N_2 \in \mathbb{N}$ and $N_2 \geq N_1$, then
		\[ \mynorm{\sum_{n=N_1}^{N_2} 2^{-n} f_n^{2}}  \leq \sum_{n=N_1}^{N_2} 2^{-n} \mynorm{f_n^{2}} \leq \sum_{n=N_1}^{N_2} 2^{-n}. \]
		Since $C(X)$ is complete, we have $f \in C(X)$. Let $\lambda \in (0, 1)$. By (\ref{existence-strictly-convex-function-eq1}), $f_i(x_1) \neq f_i(x_2)$ and using the strict convexity of the squared function, then
		\begin{equation*}
			\begin{split}
				f_i^{2}(\lambda x_1 + (1-\lambda x_2)) &= (\lambda f_i(x_1) + (1- \lambda) f_i(x_2))^{2} \\
				&< \lambda f_i^{2}(x_1) + (1-\lambda) f_i^{2}(x_2).
			\end{split}
		\end{equation*}
		Since $f_n^{2}(\lambda x_1 + (1-\lambda x_2)) \leq \lambda f_n^{2}(x_1) + (1-\lambda) f_n^{2}(x_2)$ for any $n \in \mathbb{N}$, we have that $f$ is strictly convex.
	\end{proof}
	
	We shall now also require the following result about extreme points \cite[Proposition 1.3]{phelps}.
	
	\begin{proposition}
		If $X$ is a metrizable compact convex subset of a topological vector space, then the set of extreme points is Borel.
	\end{proposition}
	
	We can now prove Choquet's theorem.
	\begin{theorem}[Choquet]\label{choquet-theorem}
		Suppose that $X$ is a metrizable compact convex subset of a locally convex space E, and that $x_0$ is an element of $X$. Then there is a probability measure $\mu$ on $X$ which represents $x_0$ and is supported by the extreme points of $X$.
	\end{theorem}
	\begin{proof}
		Let $\mathcal{A} \subseteq C(X, \mathbb{R})$ be the set of affine functions.   By Lemma \ref{existence-strictly-convex-function}, there exists a strictly convex function $f \in C(X, \mathbb{R})$. Let $\phi :C(X, \mathbb{R}) \rightarrow \mathbb{R}$ be a functional such that $\phi(g) = \overline{g}(x_0)$ for $g \in C(X, \mathbb{R})$. Then by (I) of Proposition \ref{upper-envelope-properties}, $\phi$ is subadditive. On the other hand, let $\tau$ be defined on the subspace $\mathcal{A} + \mathbb{R}f$ of $C(X, \mathbb{R})$ such that 
		\[  \tau(h +r f) = h(x_0) + r \overline{f}(x_0), \quad h \in \mathcal{A}, r \in \mathbb{R}.  \]
		
		\par 
		We shall show that $\tau$ is dominated by $\phi$ so that we can extend $\tau$ to $C(X, \mathbb{R})$ by the Hahn-Banach theorem . If $r_1 \geq 0$ and $h \in \mathcal{A}$, then $\tau(h + r_1 f) = h(x_0) + r_1 \overline{f}(x_0) = h(x_0) + \overline{r_1 f}(x_0)  = \overline{h + r_1 f}(x_0) = \phi(h + r_1 f)$. If $ r_2 < 0$. Then $r_2 f$ is concave. Therefore, $\tau(h + r_2 f) = h(x_0) + r_2 \overline{f}(x_0) \leq h(x_0) + r_2 f(x_0) = h(x_0) + \overline{r_2 f}(x_0) = \overline{h + r_2 f}(x_0) = \phi(r + \alpha f)$, so $\tau$ can be extended to a linear functional $\rho_{\tau}$ on $C(X, \mathbb{R})$. 
		\par 
		Let $l \in C(X, \mathbb{R})$ be such that $l \leq 0$, then $\phi(l)  = \overline{l}(x_0) \leq 0$ which implies $ \rho_{\tau}(l) \leq \phi(l) \leq 0$, so $\rho_{\tau}$ is a positive linear functional. By using the Riesz Representation Theorem for $C(X, \mathbb{R})$, we obtain a Radon measure $\mu$ such that $\rho_{\tau}(f) = \int_{X} f \,  {\rm d} \mu $. Moreover, $1 = \rho_{\tau}(1) = \mu(X)$, so $\mu$ is a probability measure. Furthermore, if $\sigma$ is a continuous linear functional  on $E$, then $\sigma_{\vert X} \in \mathcal{A}$ which implies that $\int_{X} \sigma \,  {\rm d} \mu = \rho_{\tau}(\sigma) = \tau(\sigma) = \sigma(x_0)$. Therefore, $\mu$ represents $x_0$. 
		\par 
		We claim this is supported by the set $F = \lbrace x : f(x) = \overline{f}(x) \rbrace$ which we show is a subset of the extreme points of $X$. Since $f$ and $\overline{f}$ are measurable, it follows that $F$ is a Borel set. Let $x \in F$,  $y, z \in X$  and $\lambda \in (0, 1)$, such that $ x = \lambda y + (1 - \lambda) z$ and $y \neq z$. Since $f$ is strictly convex, we have
		\[ f(x)  < \lambda f(y) + (1- \lambda) f(z) \leq \overline{f}(\lambda y + (1 - \lambda) z ) = \overline{f}(x), \]
		which is a contradiction, so $x$ is an extreme point of $X$. Let us now show that $\mu(X/F) = 0$. Since $f \leq \overline{f}$, then $\rho_{\tau}(f) \leq \rho_{\tau}(\overline{f})$. Recall that $h \in \mathcal{A}$ and suppose that $h \geq f$. Therefore, $h \geq \overline{f}$ and $h(x_0) = \rho_{\tau}(h) \geq \rho_{\tau}(\overline{f})$. Hence, $\inf \lbrace h(x_0) : h \geq f, h \in \mathcal{A} \rbrace \geq \rho_{\tau}(\overline{f})$, so $ \rho_{\tau}(f) = \overline{f}(x_0) \geq \rho_{\tau}(\overline{f})$. This implies that $\rho_{\tau}(f) = \rho_{\tau}(\overline{f})$ which allows us to conclude
		\[  \int_{X/F} (\overline{f} - f) \, {\rm d} \mu  = 0 \implies \mu(X/F) = 0, \]
		so $\mu$ is supported by the set of extreme points which proves the theorem.
	\end{proof}
	
		\section{Operator Monotonicity And Divided Differences}
	In this section, we introduce operator monotone functions and their relationship with operator convex functions. Specifically, we demonstrate that if $f$ is operator convex and twice continuously differentiable on $(a, b) \subseteq \mathbb{R}$, then 
	\begin{equation}\label{divided-difference-function-definition}
		F_{t_1}(t) = 
		\begin{cases} 
			\frac{f(t) - f(t_1)}{t - t_1} &  t \neq t_1 \\
			f'(t_1) & t=t_1
		\end{cases}
	\end{equation}
	is operator monotone on $(a, b)$ for $t_1 \in (a, b)$. For notational convenience, we denote by $L_{t_1}$ the transformation
	\begin{equation}\label{lowner-transformation-to-divided-difference}
		L_{t_1}: f \rightarrow F_{t_1}.
	\end{equation}
	
	\subsection{Operator Monotone and Operator Convex functions}
	We shall define what we mean by operator monotone and operator convex/concave functions and we show that if $f$ is operator monotone, then $f$ is operator concave. Let us first define what we mean by operator monotone and operator convex of order $n$ where $n \in \mathbb{N}$.
	\begin{definition}
		Let $n \in \mathbb{N}$. A function $f: (a, b) \rightarrow \mathbb{R}$ on $(a, b) \subseteq \mathbb{R}$ is said to be operator monotone of order $n$ if for any $A \leq B$, then $f(A) \leq f(B)$ 	where $A, B$ are $n \times n$ Hermitian matrices and $\textnormal{sp}(A) \cup \textnormal{sp}(B) \subseteq (a,b)$.
	\end{definition}
	
	\begin{definition}
		Let $n \in \mathbb{N}$. A function $f: (a,b) \rightarrow \mathbb{R}$ on $(a, b) \subseteq \mathbb{R}$ is said to be operator convex of order $n$ if
		\[ f(tA + (1-t)B) \leq t f(A) + (1-t)f(B), \quad \forall \, t \in [0,1]\]
		where $A, B$ are elements of a $n \times n$ Hermitian matrices,  and $\textnormal{sp}(A) \cup \textnormal{sp}(B) \subseteq (a,b)$. We say that $f$ is operator concave if $-f$ is operator convex.
	\end{definition}
	Note that for any $t \in [0,1]$, $\text{sp}(tA + (1-t)B) \subseteq (a,b)$ which follows from the fact that \[ \textnormal{min}(\textnormal{sp}(A) \cup \textnormal{sp}(B))I \leq tA + (1-t)B \leq \textnormal{max}(\textnormal{sp}(A) \cup \textnormal{sp}(B))I.  \]
	This brings us to the first result \cite[Theorem 2.1]{hansen-fasttrack}.
	\begin{theorem}\label{operator-monotone-implies-operator-concave}
		If a function $f: \mathbb{R}^{+} \rightarrow \mathbb{R}$  is operator monotone of order $2n$, then $f$ is operator concave of order $n$. In particular, since $f$ is scalar concave on $\mathbb{R}^{+}$, it follows that $f$ is continuous.
	\end{theorem}
	\begin{proof}
		Let $A, B$ be $n \times n$ Hermitian matrices with spectrum in $\mathbb{R}^{+}$. Let $\lambda \in [0,1]$. Denote the $n \times n$ identity matrix by $\mathbb{I}_n$ and the $n \times n$ zero matrix by $0_n$. Consider the $2n \times 2n$ unitary matrix $U$ such that
		\[  U  =  		
		\begin{pmatrix}
			\lambda^{1/2} \mathbb{I}_n &   -(1-\lambda)^{1/2}\mathbb{I}_n \\
			(1-\lambda)^{1/2} \mathbb{I}_n & \lambda^{1/2} \mathbb{I}_n
		\end{pmatrix}. \]
		Furthermore,  
		\begin{align*}
			U^{*} 	\begin{pmatrix}
				A &   0_n \\
				0_n & B
			\end{pmatrix}  U 
			&= 
			\begin{pmatrix}
				\lambda^{1/2} \mathbb{I}_n & (1-\lambda)^{1/2} \mathbb{I}_n  \\
				-(1-\lambda)^{1/2}\mathbb{I}_n   & \lambda^{1/2} \mathbb{I}_n
			\end{pmatrix} \\
			&*
			\begin{pmatrix}
				\lambda^{1/2} A &   -(1-\lambda)^{1/2} A  \\
				(1-\lambda)^{1/2} B & \lambda^{1/2} B
			\end{pmatrix} \\
			&=
			\begin{pmatrix}
				\lambda A + (1-\lambda)B &   \lambda^{1/2}(1-\lambda)^{1/2}(B-A) \\
				\lambda^{1/2}(1-\lambda)^{1/2}(B-A) & (1-\lambda)A + \lambda B
			\end{pmatrix}. 
		\end{align*}
		
		Let $C = -\lambda^{1/2}(1-\lambda)^{1/2}(B-A), t \geq \max \lbrace \mynorm{A}, \mynorm{B} \rbrace$ and $\epsilon > 0$. It is clear that $t \geq \mynorm{(1-\lambda)A + \lambda B}$ which implies that $2t - \mynorm{(1-\lambda)A + \lambda B} \geq t$. Therefore, 
		\begin{align*}
			\begin{pmatrix}
				\lambda A + (1-\lambda)B + \epsilon \mathbb{I}_n &  0_n\\
				0_n & 2t \mathbb{I}_n
			\end{pmatrix} - U^{*} 	
			\begin{pmatrix}
				A &   0_n \\
				0_n & B
			\end{pmatrix}  U \geq 
			\begin{pmatrix}
				\epsilon \mathbb{I}_n  &  C  \\
				C & t \mathbb{I}_n
			\end{pmatrix} = \mathcal{C}_{\epsilon, t}.
		\end{align*}
		By using Schur complements, we will establish necessary conditions for $\mathcal{C}_{\epsilon, t}$ to be positive semi-definite.
		\begin{align*}
			&\begin{pmatrix}
				\mathbb{I}_n  &  t^{-1}C  \\
				0_n & \mathbb{I}_n
			\end{pmatrix} 
			\begin{pmatrix}
				\epsilon\mathbb{I}_n - t^{-1}C^{2} &  0_n \\
				0_n & t \mathbb{I}_n
			\end{pmatrix} 
			\begin{pmatrix}
				\mathbb{I}_n &  t^{-1}C \\
				0_n  & \mathbb{I}_n
			\end{pmatrix}^{*} \\
			& = 
			\begin{pmatrix}
				\mathbb{I}_n  &  t^{-1}C  \\
				0_n & \mathbb{I}_n
			\end{pmatrix} 
			\begin{pmatrix}
				\epsilon\mathbb{I}_n - t^{-1}C^{2} &  0_n \\
				0_n & t \mathbb{I}_n
			\end{pmatrix} 
			\begin{pmatrix}
				\mathbb{I}_n &  0_n \\
				t^{-1}C & \mathbb{I}_n
			\end{pmatrix} \\
			& = 
			\begin{pmatrix}
				\mathbb{I}_n  &  t^{-1}C  \\
				0_n & \mathbb{I}_n
			\end{pmatrix} 
			\begin{pmatrix}
				\epsilon\mathbb{I}_n - t^{-1}C^{2} &  0_n \\
				C & t \mathbb{I}_n
			\end{pmatrix} \\
			& = 
			\begin{pmatrix}
				\epsilon \mathbb{I}_n  &  C  \\
				C & t \mathbb{I}_n
			\end{pmatrix}.
		\end{align*}
		Therefore, if $ \epsilon\mathbb{I}_n - t^{-1}C^{2} \geq0 $ or $ \epsilon \geq t^{-1}\mynorm{C}^{2}$, then $\mathcal{C}_{\epsilon, t}$ is positive semi-definite. If $t$ is large enough to also satisfy this inequality, we obtain
		\begin{align}\label{monotone-is-concave-eq1}
			\begin{pmatrix}
				\lambda A + (1-\lambda)B + \epsilon \mathbb{I}_n & 0_n \\
				0_n & 2t \mathbb{I}_n
			\end{pmatrix} \geq 
			U^{*} 	
			\begin{pmatrix}
				A &   0_n \\
				0_n & B
			\end{pmatrix}  U. 
		\end{align} 
		
		Since the  spectrum of $A, B$ consists of a finite number of points, then we can use polynomial interpolation to write $f$ as a polynomial.   Then for any unitary matrix $U$, $f(U^{*}AU) = U^{*}f(A)U$. Moreover, since for any block diagonal matrix
		\[ 	
		\begin{pmatrix}
			A &  0_n  \\
			0_n &  B
		\end{pmatrix}^{k} 
		= 	 	
		\begin{pmatrix}
			A^{k} &  0_n  \\
			0_n &  B^{k}
		\end{pmatrix}, 
		\]
		we can conclude that
		\begin{equation*}
			f\begin{pmatrix}
				A &  0_n  \\
				0_n &  B
			\end{pmatrix}
			= 	 	
			\begin{pmatrix}
				f(A) &  0_n  \\
				0_n &  f(B)
			\end{pmatrix}.
		\end{equation*} 
		Therefore, 
		\begin{equation}\label{monotone-is-concave-eq3}
			\begin{split}
				&f \bigg(  U^{*} 	\begin{pmatrix}
					A &   0_n \\
					0_n & B
				\end{pmatrix}  U \bigg) = U^{*} \begin{pmatrix}
					f(A) &   0_n \\
					0_n & f(B)
				\end{pmatrix} U \\
				&= \begin{pmatrix}
					\lambda f(A) + (1- \lambda)f(B) &   \lambda^{1/2}(1- \lambda)^{1/2}(f(B) - f(A))  \\
					\lambda^{1/2}(1- \lambda)^{1/2}(f(B) - f(A)) & (1- \lambda)f(A) + \lambda f(B)
				\end{pmatrix}.
			\end{split}
		\end{equation}
		By combining (\ref{monotone-is-concave-eq1}) and (\ref{monotone-is-concave-eq3}) and using the fact that $f$ is operator monotone of order $2n$, 
		we obtain
		\begin{align*}
			&\begin{pmatrix}
				\lambda f(A) + (1- \lambda)f(B) &   \lambda^{1/2}(1- \lambda)^{1/2}(f(B) - f(A))  \\
				\lambda^{1/2}(1- \lambda)^{1/2}(f(B) - f(A)) & (1- \lambda)f(A) + \lambda f(B)
			\end{pmatrix} \\  	
			&\leq \begin{pmatrix}
				f(\lambda A + (1-\lambda)B + \epsilon \mathbb{I}_n) &  0_n\\
				0_n & f(2t \mathbb{I}_n)
			\end{pmatrix}.
		\end{align*} 		
		Therefore,
		\begin{equation}\label{monotone-is-concave-eq4}
			\lambda f(A) + (1-\lambda) f(B) \leq f(\lambda A + (1-\lambda) B + \epsilon\mathbb{I}_n), \quad \forall\, \epsilon > 0.
		\end{equation}
		Let $x_1, x_2 > 0$. Since $f(\delta \mathbb{I}_n) = f (\delta) \mathbb{I}_n$ for any $\delta > 0$,   the above inequality implies that
		\begin{align}\label{monotone-is-concave-eq5}
			\lambda f(x_1)+ (1-\lambda) f(x_2) \leq f(\lambda x_1 + (1-\lambda) x_2 + \epsilon), \quad \forall\, \epsilon > 0.
		\end{align}
		Since $f$ is monotone, we can define the right hand limit $f^{+}$ by setting 
		\[ f^{+}(x) = \lim_{\epsilon \rightarrow 0+} f(x + \epsilon), \quad x > 0. \]
		By definition of $f^{+}$, we have
		\begin{equation}\label{monotone-is-concave-eq6}
			\lambda f^{+} (x_1) + (1-\lambda) f^{+} (x_2) \leq \lambda f(x_1 + \epsilon) + (1-\lambda) f(x_2 + \epsilon), \quad \forall \, \epsilon > 0.
		\end{equation}
		By combining, (\ref{monotone-is-concave-eq5}) and (\ref{monotone-is-concave-eq6}), we can conclude that
		\begin{equation*}
			\lambda f^{+} (x_1) + (1-\lambda) f^{+} (x_2) \leq f(\lambda x_1 + (1-\lambda)x_2 + 3 \epsilon), \quad \forall \, \epsilon > 0.
		\end{equation*}
		As $\epsilon \rightarrow 0$ in the above, we obtain
		\[ \lambda f^{+}(x_1) + (1-\lambda) f^{+}(x_2) \leq f^{+}(\lambda x_1 + (1-\lambda) x_2). \]
		Therefore, $f^{+}$ is concave on $(0, \infty)$ which implies it is continuous. Since $f$ is monotone increasing, it follows that 
		\[ f^{+}(t- \epsilon) \leq f(t) \leq f^{+}(t),  \quad 0 < \epsilon < t. \]
		As $\epsilon \rightarrow 0$ in the above, we obtain $f^{+} = f$ which implies that $f$ is continuous. By \cite[Proposition 5.3.2]{kadison-vol-1}, $f$ is s.o.t continuous on bounded sets of self-adjoint operators, so as $\epsilon \rightarrow 0$ in (\ref{monotone-is-concave-eq4}),  $\lambda f(A) + (1-\lambda) f(B) \leq f(\lambda A + (1-\lambda) B)$.
	\end{proof}
	
	Let us now define what we mean by an operator monotone function.
	\begin{definition}
		\index{Operator!monotone} A function $f: (a, b) \rightarrow \mathbb{R}$ on $(a, b) \subseteq \mathbb{R}$ is said to be operator monotone if for any $A \leq B$, then $f(A) \leq f(B)$ 	where $A, B$ are elements of a $C^{*}$-algebra, and $\textnormal{sp}(A) \cup \textnormal{sp}(B) \subseteq (a,b)$.
	\end{definition}
	It is natural to ask whether operator monotone functions can be characterized by being operator monotone of order $n$ for any $n \in \mathbb{N}$ \cite[Theorem 2.5]{bendat}.
	
	\begin{theorem}\label{operator-monotone-characterized-by-finite-orders}
		A function $ f: (a,b) \rightarrow \mathbb{R} $ on $(a, b) \subseteq \mathbb{R}$ is operator monotone iff  $f$ is operator monotone of order $n$ for any $n \in \mathbb{N}$.
	\end{theorem}
	\begin{proof}
		Suppose $f$ is operator monotone. Since an $n \times n$ Hermitian matrix is a bounded operator from $\mathbb{C}^{n}$ to $\mathbb{C}^{n}$, then it is clear that an operator monotone function is operator monotone of order $n$ for any $n \in \mathbb{N}$. 
		\par
		Suppose $f$ is operator monotone of order $n$ for any $n \in \mathbb{N}$. Let $A \leq B$ and $\textnormal{sp}(A) \cup \textnormal{sp}(B) \subseteq (a,b)$.  Consider the map $\alpha: \mathscr{H} \rightarrow \mathscr{H}$ such that
		\[  \alpha(\lbrace x_1, \ldots, x_n \rbrace) = \lbrace x_1, \ldots, x_n, Ax_1, \ldots, Ax_n, Bx_1, \ldots, Bx_n \rbrace. \]
		Let $D = \lbrace \alpha(\lbrace x_1, \ldots, x_n \rbrace): \lbrace x_1, \ldots, x_n \rbrace \subseteq \mathscr{H} \rbrace$ be a directed set under the usual order of subset inclusion. Consider the map $\gamma: D \rightarrow $ \mybh\, such that $\gamma(d) = P_d$ where $P_d$ is the projection on the subspace generated by $d \in D$.
		Naturally, $P_{\gamma(d)} A P_{\gamma(d)} \leq P_{\gamma(d)} B P_{\gamma(d)} $. Furthermore, since $P_{\gamma(d)}  B(\mathscr{H}) P_{{\gamma(d)} }$ is $\ast$-isomorphic to a finite direct sum of matrix algebras over $\mathbb{C}$, we have
		\begin{equation}\label{operator-monotone-characterized-by-finite-orders-eq1}
			f (P_{\gamma(d)} A P_{\gamma(d)}) \leq  f (P_{\gamma(d)} B P_{\gamma(d)}).
		\end{equation}
		It is clear that $ \displaystyle{\lim_{d \in D} P_{\gamma(d)} A P_{\gamma(d)} \rightarrow A}$ and $ \displaystyle{\lim_{d \in D} P_{\gamma(d)} B P_{\gamma(d)} \rightarrow B}$ in s.o.t.
		By Theorem \ref{operator-monotone-implies-operator-concave} and \cite[Proposition 5.3.2]{kadison-vol-1}, $f $ is s.o.t continuous on bounded sets of self-adjoint operators, so $f(A) \leq f(B)$ as we take limits in (\ref{operator-monotone-characterized-by-finite-orders-eq1}).
	\end{proof}
	
	This leads us to defining operator convex functions.
	\begin{definition}
		\index{Operator!convex} \index{Operator!concave}  A  function $f: (a,b) \rightarrow \mathbb{R}$ on $(a, b) \subseteq \mathbb{R}$ is said to be operator convex if
		\[ f(tA + (1-t)B) \leq t f(A) + (1-t)f(B), \quad \forall \, t \in [0,1]\]
		where $A, B$ are elements of a $C^{*}$-algebra,  and $\textnormal{sp}(A) \cup \textnormal{sp}(B) \subseteq (a,b)$. Furthermore, $f$ is said to be operator concave if $-f$ is operator convex.
	\end{definition}
	
	Similar to Theorem \ref{operator-monotone-characterized-by-finite-orders}, we have the following theorem \cite[Lemma 3.1]{bendat}.
	\begin{theorem}\label{operator-concave-characterized-by-finite-orders}
		A continuous function $ f: (a,b) \rightarrow \mathbb{R} $ on $(a, b) \subseteq \mathbb{R} $ is operator convex iff  $f$ is operator convex of order $n$ for any $n \in \mathbb{N}$.
	\end{theorem}
	The proof is omitted due to its similarity with the proof of Theorem \ref{operator-monotone-characterized-by-finite-orders}. Finally, we arrive at the theorem we set out to prove in this subsection. 
	\begin{theorem}\label{operatpr-monotone-implies-operator-concave}
		If a function $f: \mathbb{R}^{+} \rightarrow \mathbb{R}$ is operator monotone, then $f$ is operator concave.
	\end{theorem}
	\begin{proof}
		If $f$ is operator monotone, then $f$ is operator monotone of order $n$ for any $n \in \mathbb{N}$ by Theorem \ref{operator-monotone-characterized-by-finite-orders}. By Theorem \ref{operator-monotone-implies-operator-concave}, $f$ is operator concave of order $n$ for any $n \in \mathbb{N}$. Consequently, Theorem \ref{operator-concave-characterized-by-finite-orders} allows us to conclude that $f$ is operator concave.
	\end{proof}
	
	\subsection{The Chain Rule and Divided Differences}
	The goal of this section is to use the chain rule for matrices and the concept of divided differences to show that if $f: (a, b) \rightarrow \mathbb{R}$ on $(a, b) \subseteq \mathbb{R}$ is operator convex and twice continuously differentiable, then $L_{t_1}f$ (recall (\ref{lowner-transformation-to-divided-difference})) is operator monotone on $(a, b)$ for $t_1 \in (a, b)$.  Let $f: (a, b) \rightarrow \mathbb{R}$ be continuously differentiable and $t_2 \in (a, b)$. Then the first divided difference $\Delta f(t_1, t_2)$ is defined as
	\[
	\Delta f (t_1,t_2) = 
	\begin{cases} 
		\frac{f(t_2) - f(t_1)}{t_2 - t_1} &  t_1 \neq t_2 \\
		f'(t_1) & t_1=t_2
	\end{cases}.
	\]
	Let $\lbrace t_1, \ldots, t_n \rbrace \subseteq (a, b)$. Then we can define the corresponding $n \times n$ \emph{matrix of first divided differences} as  
	\[ [\Delta f(t_i, t_j)]_{i, j = 1}^{n}  = \begin{bmatrix}
		\Delta f(t_1, t_1) & \ldots & \Delta f(t_1, t_n) \\
		\vdots & \ddots  & \vdots \\
		\Delta f(t_n, t_1) & \ldots & \Delta f(t_n, t_n)
	\end{bmatrix}.\]
	\noindent
	It can be easily seen that the above matrix is symmetric. In addition, we shall also consider the second divided difference. Let $f: (a, b) \rightarrow \mathbb{R}$ be a twice continuously differentiable function on $(a, b) \subseteq \mathbb{R}$. Let $t_1, t_2, t_3 \in (a, b)$  be distinct. Then the second divided difference is defined as
	\begin{equation}\label{second-divided-difference-equation}
		\Delta^{2} f(t_1, t_2, t_3) = \frac{\Delta f(t_3, t_1) - \Delta f(t_2, t_1) }{t_3 - t_2}.
	\end{equation}
	For repeated values, $\Delta^{2} f(t_1, t_2, t_2) = \lim_{t \rightarrow t_2} \Delta^{2} f(t_1, t_2, t)$. In particular,
	\[ \Delta^{2} f(t_1, t_1, t_1)  = \lim_{t \rightarrow t_1} \Delta^{2} f(t_1, t_2, t) \\ = \lim_{s \rightarrow t_1} \lim_{t \rightarrow t_1} \Delta^{2} f(t_1, s, t). \]
	This leads us to the chain rule for matrices \cite[Theorem 6.6.30]{horn-book}. However, we require the following  clarifications first.
	\begin{enumerate}[label=\textnormal{(\Roman*)}]
		\item The derivative of an $n \times n$ matrix is the entrywise derivative. 
		\item $H_n(a,b)$ denotes the set of $n \times n$ Hermitian matrices with spectrum in the open real interval $(a, b)$.
		\item Let $A, B$  be arbitrary $n \times n$ matrices. Then $A \circ B$ denotes the Schur product (or entrywise multiplication).
	\end{enumerate}
	
	\begin{theorem}[Chain rule]\label{chain-rule-for-matrices}
		Let $f :(a, b) \rightarrow \mathbb{R}$ and $\gamma: (c,d) \rightarrow H_n(a, b)$ be functions on $(a, b), (c,d) \subseteq \mathbb{R}$ such that $\textnormal{sp}(\gamma(t)) \subseteq (a,b)$ for all $t \in (c, d)$. Let $\gamma(t) = U(t) D(t) U(t)^{*}$, where $U$ is a unitary matrix, and $D=\textnormal{diag}(\lambda_1(t),  \ldots, \lambda_n(t))$ for all $t \in (c,d)$. Then, the following hold
		\begin{enumerate}[label=\textnormal{(\Roman*)}]
			\item If $f$ and $\gamma$ are continuously differentiable, then $f \circ \gamma$ is continuously differentiable with derivative 
			\[ \frac{d}{dt} f(\gamma(t)) = U(t) \bigg( [\Delta f(\lambda_i(t), \lambda_j(t))]_{i, j = 1}^{n} \circ [U(t)^{*} \gamma'(t)U(t)] \bigg) U(t)^{*}  \]
			for all $t \in (c, d)$.
			\item  If $f$ and $\gamma$ are twice continuously differentiable, then $f \circ \gamma$ is twice continuously differentiable. Let $c_k(t)$ denote the $k$'th column of $U(t)^{*}\gamma'(t)U(t)$, then the second derivative of $f \circ \gamma$ is
			\begin{equation*}
				\begin{split}
					\frac{d^{2}}{dt^{2}} f(\gamma(t)) = U(t) &\bigg( 2 \sum_{k=1}^{N} [\Delta^{2}f(\lambda_i(t), \lambda_j(t), \lambda_k(t))]^{n}_{i,j=1} \circ (c_k(t) c_k^{*}(t))   \\
					&+[\Delta f(\lambda_i(t), \lambda_j(t))]_{i, j = 1}^{n} \circ (U(t)^{*} \gamma''(t) U(t)) \bigg)U(t)^{*}
				\end{split}
			\end{equation*}
			for all $t \in (c, d)$.
		\end{enumerate}
	\end{theorem}
	
	An application of the chain rule allows us to characterize operator monotone functions of order $n$. This is given in the next theorem \cite[Thoerem 6.6.36]{horn-book}.
	\begin{theorem}\label{first-divided-difference-positive-implies-matrix-monotone}
		Let $f:(a,b) \rightarrow \mathbb{R}$ be a continuously differentiable function on $(a, b ) \subseteq \mathbb{R}$. Then $f$ is operator monotone of order $n$ iff 
		\[  [\Delta f(t_i, t_j)]_{i,j=1}^{n} \geq 0, \quad \forall \, t_1, \ldots,  t_n \in (a, b).\]
	\end{theorem}
	\begin{proof}
		Suppose the matrix of first divided differences is positive semi-definite. Let $A, B \in H_n(a, b)$ and $A \leq B$. Let $\epsilon_1 > 0$ such that $\text{sp}((1-t)A + tB) \subseteq (a, b)$ for $t \in (-\epsilon_1, 1 + \epsilon_1)$. Consider $\gamma(t) = (1-t)A + tB$ for $t \in (-\epsilon_1, 1 + \epsilon_1)$. Furthermore, $\gamma'(t) = B-A$ for $t \in (-\epsilon_1, 1 + \epsilon_1)$, so by (I) of Theorem \ref{chain-rule-for-matrices}, we have 
		\begin{align*}
			\begin{split}
				&f(B)-f(A) = \int_{0}^{1} \frac{d}{dt} f(\gamma(t)) \, \rm{d}t \\
				&= \int_{0}^{1} U(t) \bigg( [\Delta f(\lambda_i(t), \lambda_j(t))]_{i,j=1}^{n} \circ (U(t)^{*} (B-A) U(t)) \bigg)U(t)^{*} \, dt.
			\end{split}
		\end{align*}
		Since  $[\Delta f(\lambda_i(t), \lambda_j(t))]_{i,j=1}^{n}  \geq 0$ and $U(t)^{*} (B-A) U(t) \geq 0$, the Schur product is positive semi-definite as well. This implies that the integral is the limit of sums of positive semi-definite matrices, so $f(A) \leq f(B)$.
		\par 
		Suppose $f$ is operator monotone of order $n$. Let $x = (x_1, \ldots, x_n) \in \mathbb{C}^{n}$ and let $D = \text{diag}(t_1, \ldots, t_n)$ where $\lbrace t_1, \ldots, t_n \rbrace \subseteq (a, b)$. Let $\epsilon_2 > 0$ such that $A = D - \epsilon_2 (xx^{*})$ and  $B  = D + \epsilon_2 (xx^{*})$ are elements of $H_n(a,b)$. Let $\gamma(t) = (1-t) A + tB$ for $t \in (0,1)$. Let us now apply (I) of Theorem \ref{chain-rule-for-matrices} for $f'(\gamma(1/2))$.  Since  $\gamma(1/2)$ is a diagonal matrix, it follows that $U(1/2) = \mathbb{I}_n$. Furthermore,  $\gamma^{'}(1/2)= 2 \epsilon_2 (x x^{*})$. These two observations imply that
		\begin{equation*}\label{horn-th-6.6.36-eq1}
			(f \circ \gamma)'(1/2)  = \frac{d}{dt} f(\gamma(t)) \vert_{t=1/2} = [\Delta f(\lambda_i, \lambda_j)]_{i,j=1}^{n} \circ 2 \epsilon_2(x x^{*}).
		\end{equation*}

		
		Moreover, if $t \geq s$ for $t, s \in (0,1)$, then 
		\[  \gamma(t) - \gamma(s) =  (1-t)A + tB - (1-s)A + sB   = (t-s)(B-A).  \]
		However, $B-A = 2 \epsilon_2 xx^{*}$ is positive semi-definite since $\myinnerproduct{(xx^{*})y}{y} = y^{*}xx^{*}y = (y^{*}x)(y^{*}x)^{*} = \myabs{y^{*}x}^{2}$. Thus,  $\gamma(t) \geq \gamma(s)$ and since $f$ is operator monotone of order $n$, then $f(\gamma(t)) \geq f(\gamma(s))$. This allows us to conclude that
		\[ \frac{d}{dt} f(\gamma(t)) \vert_{t=1/2} = \lim_{h \rightarrow 0+} \frac{f(\gamma(1/2 + h)) - f(\gamma(1/2))}{h} \]
		is positive semi-definite. Consider $y = (1, \ldots, 1) \in \mathbb{R}^{n}$. Then 
		\begin{equation}\label{first-divided-difference-positive-implies-matrix-monotone-eq1}
			0 \leq \bigg \langle \frac{(f \circ \gamma)'(1/2)}{2 \epsilon_2} y, y  \bigg \rangle . 
		\end{equation}
		Computing $y^{*}([\Delta f(\lambda_i, \lambda_j)]_{i,j=1}^{n} \circ (x x^{*}))y$, we get
		\begin{align*}
			\begin{pmatrix}
				1, \ldots, 1
			\end{pmatrix} 
			&\begin{bmatrix}
				\Delta f(\lambda_1, \lambda_1) x_1 x^{*}_1 & \Delta f(\lambda_1, \lambda_2) x_1 x^{*}_2 & \ldots & \Delta f(\lambda_1, \lambda_n) x_1 x^{*}_n \\
				\vdots & \vdots & \vdots & \vdots \\
				\Delta f(\lambda_n, \lambda_1) x_n x^{*}_1 & \Delta f(\lambda_n, \lambda_2) x_n x^{*}_2 & \ldots & \Delta f(\lambda_n, \lambda_n) x_n x^{*}_n
			\end{bmatrix}
			\begin{pmatrix}
				1 \\
				\vdots \\
				1
			\end{pmatrix} 
			\\
			&= 		
			\begin{pmatrix}
				1, \ldots, 1
			\end{pmatrix} 
			\begin{pmatrix}
				x_1 \sum_{i=1}^{n} \Delta f(\lambda_1, \lambda_i) x^{*}_i \\
				\vdots \\
				x_n \sum_{i=1}^{n} \Delta f(\lambda_n, \lambda_i) x^{*}_i
			\end{pmatrix} 
			= \sum_{j=1}^{n} x^{*}_j \sum_{i=1}^{n} \Delta f(\lambda_j, \lambda_i) x_i
			\\
			&= \myinnerproduct{[\Delta f(\lambda_i, \lambda_j)]_{i,j=1}^{n} x}{x}.
		\end{align*}		
		Therefore, by combining the above and (\ref{first-divided-difference-positive-implies-matrix-monotone-eq1}), we conclude that \[[\Delta f(\lambda_i, \lambda_j)]_{i,j=1}^{n} \circ (x x^{*}) \] is positive semi-definite.
	\end{proof}
	
	Let us now proceed to prove a relationship between the matrix of second divided differences and convex matrix functions \cite[Page 544]{horn-book}. 
	
	\begin{lemma}\label{lemma-for-second-difference}
		Let $f:(a, b) \rightarrow \mathbb{R}$ be a twice continuously differentiable function on $(a, b) \subseteq \mathbb{R}$. Let $A, B \in H_n(a, b)$ and $\gamma(t) = (1-t)A + tB$ for all $t \in (0,1)$. If $f$ is operator convex of order $n$, then
		\[  (f \circ \gamma)''(t) \geq 0, \quad \forall \, t \in (0,1).  \]
	\end{lemma}
	\begin{proof}
		Let $x \in \mathbb{C}^{n}$ and $\phi_x(t) = \myinnerproduct{f(\gamma(t))x}{x}$ for all $t \in (0,1)$. By (II) of Theorem \ref{chain-rule-for-matrices}, $f \circ \gamma$ is twice continuously differentiable, so $\phi_x$ is twice continuously differentiable. Since $f$ is operator convex of order $n$, then $\phi_x$ is a convex function. Therefore, the second derivative of $\phi_x$ is non-negative which implies that
		\begin{equation*}
			0 \leq \phi_x^{''}(t) = \bigg\langle (f \circ \gamma)''(t) x, x \bigg\rangle, \quad \forall \, t \in (0,1).
		\end{equation*}
		Since the above holds for any arbitrary $x \in \mathbb{C}^{n}$, the result follows.
	\end{proof}
	
	We are now in a position to prove \cite[Theorem 6.6.52 (1)]{horn-book}.
	
	\begin{theorem}\label{matrix-convex-implies-second-divided-difference-positive}
		Let $f: (a, b) \rightarrow \mathbb{R}$ be a twice continuously differentiable function on $(a, b) \subseteq \mathbb{R}$. If $f$ is operator convex of order $n$, then
		\[  [\Delta^{2}f(t_i, t_j, t_1)]_{i, j = 1}^{n} \geq 0, \quad \forall \,  t_1, \ldots, t_n \in (a, b).\]
	\end{theorem}
	\begin{proof}
		Let $D = \text{diag}(t_1, \ldots, t_n) \in H_n (a, b)$ and $L > 0$. Let $X$ be an $n \times n$  matrix such that
		\[ X = 
		\begin{bmatrix}
			L & 1 & \ldots & 1 \\
			1 & 0 & \ldots & 0 \\
			\vdots & \vdots & \ddots & \vdots \\
			1 & 0 & \ldots & 0
		\end{bmatrix}.
		\]
		Consider $A = D - \epsilon X$ and $B = D + \epsilon X$ such that $\epsilon > 0$ is small enough so that $A, B \in H_n(a, b)$. Let $\gamma(t) = (1-t)A + tB$ for all $t \in (0,1)$. Therefore, $\gamma(1/2) = D $ and $\gamma'(1/2) = 2 \epsilon X$. Let $c_k$ denote the $k$'th column of $X$. Let us now apply (II) of Theorem \ref{chain-rule-for-matrices} for $(f \circ \gamma)^{''}(1/2)$. Since  $\gamma(1/2)$ is a diagonal matrix, then $U(1/2) = \mathbb{I}_n$ where $\mathbb{I}_n$ denotes the $n \times n$ identity matrix. Furthermore,  $\gamma^{''}(1/2)=0$ which allows us to conclude that 
		\begin{equation*}
			(f \circ \gamma)''(1/2) = \frac{d^{2}}{d t^{2}}f(\gamma(t))\vert_{t=1/2} = 8 \epsilon^{2} \sum_{k=1}^{n} [\Delta^{2} f(t_i, t_j, t_k)]_{i,j = 1}^{n} \circ  \big( c_k c_k^{*} \big).
		\end{equation*}
		
		Let us now compute the matrices $c_k c_k^{*}$ for $1 \leq k \leq n$.
		\[  c_1 c_1^{*} =  		\begin{bmatrix}
			L^{2} & L & \ldots & L \\
			L & 1 & \ldots & 1 \\
			\vdots & \vdots & \ddots & \vdots \\
			L & 1 & \ldots & 1
		\end{bmatrix},   c_j c_j^{*} =  		\begin{bmatrix}
			1 & 0 & \ldots & 0 \\
			0 & 0 & \ldots & 0 \\
			\vdots & \vdots & \ddots & \vdots \\
			0 & 0 & \ldots & 0
		\end{bmatrix}, \quad 1 < j \leq n. \]
		Let $(x_1, \ldots, x_n) \in \mathbb{C}^{n}$ be a unit vector and $y = (x_1/L, \ldots, x_n)$. By Lemma \ref{lemma-for-second-difference}, $(f \circ \gamma)''(1/2)$ is positive semi-definite, so
		\[ 0 \leq \bigg\langle \frac{(f \circ \gamma)''(1/2)}{8 \epsilon^{2}} y,y \bigg\rangle = x^{*} [\Delta^{2} f(t_i, t_j, t_1)]_{i,j = 1}^{n} x + \sum_{k=2}^{n} \Delta^{2} f(t_1, t_1, t_k) \frac{\myabs{x_1}^{2}}{L^{2}}. \] 
		As $L \rightarrow \infty$, the result is obtained.
	\end{proof}
	
	\begin{proposition}\label{second-divided-difference-positivity-implies-first-derivative-positive}
		Let  $f: (a, b) \rightarrow \mathbb{R}$ be a twice continuously differentiable function on $(a, b) \subseteq \mathbb{R}$ and recall the transformation $L_{t_1}f$ in (\ref{lowner-transformation-to-divided-difference}), then
		\[ [\Delta^{2}f(t_i, t_j, t_1)]_{i, j = 2}^{n} = [\Delta L_{t_1}f(t_i, t_j)]_{i, j = 2}^{n},\]
		where $t_1, \ldots, t_n \in (a, b)$.
	\end{proposition}
	\begin{proof}
		Let $t_k, t_l, t_m \in (a,b)$  be distinct. We claim that $ \Delta^{2} f(t_k, t_l, t_m) = \Delta^{2} f(\pi(1), \pi(2), \pi(3))$ for any permutation $\pi$ of $t_k, t_l, t_m$. We recall (\ref{second-divided-difference-equation}) to expand $\Delta^{2}f(t_k, t_l, t_m)$.
		\begin{align*}
			\Delta^{2}f(t_k, t_l, t_m) &= \frac{(f(t_m) - f(t_k))(t_l - t_k) -  (f(t_l) - f(t_k))(t_m-t_k)}{(t_m - t_k)(t_l-t_k)(t_m - t_l)} \\
			&= \frac{f(t_m)(t_l - t_k) + f(t_k)(t_m-t_l) + f(t_l)(t_k - t_m)}{(t_m - t_k)(t_l-t_k)(t_m - t_l)} \\
			&= \frac{f(t_m)}{(t_m - t_k)(t_m - t_l)} + \frac{f(t_k)}{(t_k - t_l)(t_k - t_m)} + \frac{f(t_l)}{(t_l - t_k)(t_l - t_m)}.
		\end{align*}
		Therefore, our claim is proved. Consider $\Delta^{2} f(t_i, t_j, t_1)$ where $t_i, t_j, t_1 \in (a,b)$ and $t_i, t_j, t_1$ are not necessarily distinct. 
		\begin{align*}
			&\Delta^{2} f(t_i, t_j, t_1) = \lim_{s \rightarrow t_i} \lim_{t \rightarrow t_j} \Delta^{2} f(s, t, t_1) = \lim_{s \rightarrow t_i} \lim_{t \rightarrow t_j} \frac{\Delta f(t, t_1) - \Delta f(s, t_1) }{t-s} \\
			&= \lim_{s \rightarrow t_i} \lim_{t \rightarrow t_j} \frac{L_{t_1}f(t) - L_{t_1}f(s)}{t - s} 
			= \lim_{s \rightarrow t_i} \frac{L_{t_1}f(t_j) - L_{t_1}f(s)}{t_j - s} =  \Delta L_{t_1}f(t_i, t_j).
		\end{align*}
	\end{proof}
	
	Finally, we prove the concluding result of this section regarding the transformation in (\ref{lowner-transformation-to-divided-difference}).
	\begin{theorem}\label{concavity-implies-divided-difference-is-operator-monotone}
		If $f: (a, b) \rightarrow \mathbb{R}$ is twice continuously differentiable and operator convex on $(a, b) \subseteq \mathbb{R}$, then $L_{t_1}f$ is operator monotone for any $t_1 \in (a, b)$.
	\end{theorem}
	\begin{proof}
		By Theorem \ref{matrix-convex-implies-second-divided-difference-positive} and Proposition \ref{second-divided-difference-positivity-implies-first-derivative-positive}, the first divided difference matrix of $L_{t_1}f$ is positive semi-definite for any $n \in \mathbb{N}$. By Theorem \ref{first-divided-difference-positive-implies-matrix-monotone}, $L_{t_1}f$ is operator monotone of order $n$ for any $n \in \mathbb{N}$ which implies $L_{t_1}f$ is operator monotone by Theorem \ref{operator-monotone-characterized-by-finite-orders}.
	\end{proof}
	
		\section{L\"owner's Theorem}
	\subsection{Introduction}
	L\"owner's theorem is pivotal for the development of Kubo-Ando theory. Therefore, we shall provide a proof by following the arguments in \cite{hansen-fasttrack}. The theorem provides an integral representation for positive operator monotone functions. Namely, a positive operator monotone function $f: \mathbb{R}^{+} \rightarrow \mathbb{R}^{+}$ is a function such that $A \leq B \implies f(A) \leq f(B)$ where $A, B \in$ \mybhplusplus\,.
	We state the theorem as found in \cite[Theorem 4.9]{hansen-fasttrack}.
	\begin{theorem}\label{lowner-theorem-from-hansen}
		Let $f$ be a positive operator monotone function defined in the positive half-line. There is a  positive Radon measure $\mu$ on the closed interval $[0,1]$ such that 
		\begin{equation}\label{lowner-theory-weak}
			f(t) = \int_{[0,1]} \frac{t}{\lambda + (1-\lambda)t}  \,  {\rm d} \mu(\lambda), \quad x > 0.
		\end{equation}
		Conversely, any function given in this form is operator monotone. The measure $\mu$ is a probability measure iff $f(1) =1$.
	\end{theorem}
	Before delving into the various technicalities of the proof, we first give a general outline. Let $P$ denote the set of positive operator monotone functions and let $P_1 = \lbrace f \in P : f(1) =1 \rbrace$.
	
	\begin{enumerate}[label=\textnormal{(\Roman*)}]
		\item We use regularisation to show that  $-L_{t_1}f$ in (\ref{lowner-transformation-to-divided-difference}) is operator monotone for $f \in P$ and  $t_1 > 0$.
		\item We show that  $P_1$ is a metrizable convex compact set under a suitable topology.
		\item We use (I)  to show that the extreme points of $P_1$ are of the form $t \rightarrow \frac{t}{\lambda + (1- \lambda)t}$ on $\mathbb{R}^{+}$ for $\lambda \in [0, 1]$.
		\item We apply Choquet theory by using (II) and (III) to obtain the result.
	\end{enumerate}

	\subsection{Regularisation}
	\index{Regularisation} Let $f: \mathbb{R}^{+} \rightarrow \mathbb{R}^{+}$ be locally integrable. Therefore, we can consider its convolution with a function $\phi: \mathbb{R} \rightarrow \mathbb{R}$ which is infinitely differentiable and has the following additional properties \cite[Page 11]{donoghue}.
	\begin{enumerate}[label=\textnormal{(\Roman*)}]
		\item $\phi(x) \geq 0$ and $\phi(x) = \phi(-x)$ for $x \in \mathbb{R}^{+}$.
		\item  $\phi(x) = 0$ for $x \notin [-1, 1]$.
		\item $\int_{-1}^{1} \phi(x) \,  d x = 1$.
	\end{enumerate}
	
	Let $(a, b) \subseteq \mathbb{R}^{+}$ and $\epsilon > 0$ such that $2^{-1}(b-a) > \epsilon$. We introduce the function $\phi_{\epsilon}(x) = \frac{1}{\epsilon} \phi(\frac{x}{\epsilon})$ and we denote the convolution $f \ast \phi_{\epsilon}$ by $f_{\epsilon}$. Therefore, 
	\begin{equation}\label{regularisation-equation}
		(f \ast \phi_{\epsilon})(x) = \int^{b}_{a} f(y) \phi_{\epsilon}(x-y) \, dy  = \frac{1}{\epsilon} \int^{b}_{a} f(y)  \phi \bigg(  \frac{x-y}{\epsilon} \bigg) dy
	\end{equation}
	whenever $x \in (a + \epsilon, b - \epsilon)$. Furthermore, by substituting for $\epsilon^{-1}(x-y)$ and using the fact that $\phi$ vanishes outside $[-1, 1]$, we obtain 
	\begin{equation*}
		\int^{(x-a)/\epsilon}_{(x-b) /\epsilon} f(x- \epsilon y) \phi(y) dy = \int^{1}_{-1} f(x- \epsilon y) \phi(y) dy.
	\end{equation*}
	
	One of the well-known properties of $f_{\epsilon}$ is that $f_{\epsilon}$ is infinitely differentiable. To demonstrate this, we prove that $f_{\epsilon}$ is differentiable. The existence of higher-order derivatives follows by similar arguments.
	\begin{proposition}\label{regularisation-is-infinitely-differentiable}
		Let $f: \mathbb{R}^{+} \rightarrow \mathbb{R}^{+}$ be locally integrable and $(a, b) \subseteq \mathbb{R}^{+}$. Let $\epsilon > 0$ such that $2^{-1}(b-a) > \epsilon$, and consider 
		the regularisation $f_{\epsilon}$ given by (\ref{regularisation-equation}). Then $f_{\epsilon}$ is differentiable with derivative $f \ast \phi'_{\epsilon}$.
	\end{proposition}
	\begin{proof}
		Let $h > 0$, $x \in (a+\epsilon, b - \epsilon)$ and $\displaystyle M_1 = \sup_{x \in (a, b)}  \myabs{f(x)}$. Then 
		\begin{align*}
			\myabs{ \mathcal{E}_h(x) } &= \bigg\vert \frac{f_{\epsilon}(x+h) - f_{\epsilon}(x)  }{h} - f \ast \phi'_{\epsilon}(x) \bigg\vert \\
			&= \bigg\vert \int^{b}_{a}  \frac{ f(y) \phi_{\epsilon}(x+h -y) - f(y)\phi_{\epsilon}(x -y)}{h} -f(y)\phi'_{\epsilon}(x-y) dy \bigg\vert  \\ 
			&\leq M_1 \int^{b}_{a} \bigg\vert \frac{\phi_{\epsilon}(x+h -y)  -\phi_{\epsilon}(x-y)}{h} - \phi'_{\epsilon}(x-y) \bigg\vert dy.  
		\end{align*}
		Since we will be using Taylor's approximation, let $c_y \in (x-y, x +h -y)$ for $y \in (a, b)$. Let $\displaystyle M_2 = \sup \lbrace \myabs{\phi_{\epsilon}''(x)}: a + \epsilon - b \leq x \leq b-\epsilon + h -a\rbrace$. Then
		\begin{align*}
			\myabs{\mathcal{E}_h(x)} &\leq M_1 \int^{b}_{a} \bigg\vert  \frac{\phi_{\epsilon}'(x-y)h + 2^{-1}\phi_{\epsilon}''(c_y)h^{2}}{h} - \phi'_{\epsilon}(x-y) \bigg\vert  dy \\
			&\leq M_1 \int^{b}_{a} \bigg\vert  \frac{\phi_{\epsilon}''(c_y)h}{2} \bigg\vert dy \leq \frac{M_1 M_2 (b-a)h}{2}.
		\end{align*}
		Therefore, as $h \rightarrow 0$, the result is obtained.
	\end{proof}
	
	We recall the transformation (\ref{lowner-transformation-to-divided-difference}) which leads us to the following important theorem \cite[Theorem 3.7]{hansen-fasttrack}.
	
	\begin{theorem}\label{divided-difference-is-operator-monotone-decreasing}
		Let $f: \mathbb{R}^{+} \rightarrow \mathbb{R}^{+}$ be an operator monotone function. Then $f$ is differentiable and  $-L_{t_1}f$ is operator monotone for each $t_1 > 0$.
	\end{theorem}
	\begin{proof}
		Let  $(a, b) \subsetneq \mathbb{R}^{+}$. By Theorem \ref{operator-monotone-implies-operator-concave}, if $f$ is operator monotone, then $f$ is operator concave. Let $N \in \mathbb{N}$ be such that $2^{-1}(b-a) > N^{-1}$. Furthermore if $t_1 \in (a,b)$, then $N$ must also satisfy the property that  $t_1 \in (a + N^{-1}, b - N^{-1})$. By regularisation, we obtain a sequence $\lbrace f_n \rbrace_{n \geq N}$ of infinitely differentiable operator concave functions. We now show that the sequence $\lbrace L_{t_1}f_n \rbrace_{n \geq N}$ converges pointwise for any $t \in (a, b)$. Since $f$ is continuous,
		\[  \lim_{n \rightarrow \infty} L_{t_1}f_n(t) = \lim_{n \rightarrow \infty} \frac{f_n(t) - f_n(t_1)}{t-t_1} =  \frac{f(t) - f(t_1)}{t-t_1}, \quad t \in (a, b) \setminus \lbrace t_1 \rbrace. \]
		If $t_1 \in (a, b)$, by Proposition \ref{regularisation-is-infinitely-differentiable},
		\[ L_{t_1}f_n(t_1) = f'_n(t_1) = \int_{-1}^{1} f(t_1 - n^{-1}y) \phi'(y) dy. \]
		Since $f$ is monotone, the sequence $\lbrace L_{t_1}f_n(t_1) \rbrace_{n \geq N}$ is bounded for sufficiently large $n$. Therefore, a subsequence of $\lbrace L_{t_1}f_n \rbrace$ converges pointwise to some function $F$ for all $t \in (a, b)$.
		\par
		Let $M > 0$. For sufficiently large $n$, we have that $\myabs{L_{t_1}f_n(t)} \leq M$  for all $t \in (a, b)$. By Theorem \ref{concavity-implies-divided-difference-is-operator-monotone}, $-L_{t_1}f_n$ is operator monotone for all $n \in \mathbb{N}$, so by using Lebesgue's dominated convergence theorem and the Borel function calculus, we conclude that $-F$ is operator monotone on $(a, b)$. Continuity of $F$ implies that 
		\[  F(t_1) = \lim_{t \rightarrow t_1} \frac{f(t) - f(t_1)}{t-t_1} = f'(t_1).  \]
		Therefore, $F=L_{t_1}f$ and $-L_{t_1}f$ is operator monotone for any proper open subinterval of $\mathbb{R}^{+}$, so $-L_{t_1}f$ is operator monotone on $ \mathbb{R}^{+}$.
	\end{proof}

	Looking at the proof of the previous theorem, we can extract the following Corollary.
	\begin{corollary}\label{regularisation-gives-infinitely-differentiable-functions-converging-pointwise}
		Let $f: \mathbb{R}^{+} \rightarrow \mathbb{R}^{+}$ be an operator monotone function and $(a, b ) \subsetneq \mathbb{R}^{+}$. Then $f$ is differentiable and there is a sequence $\lbrace f_n \rbrace$ of infinitely differentiable operator monotone functions such that $\lim f_n(t) = f(t)$ for all $t \in (a, b)$. Furthermore, $\lim f'_n(t) = f'(t)$ for a finite number of points in $(a,b)$.
	\end{corollary}
	\begin{proof}
		In the previous proof, continuity of $f$ guaranteed that the sequence $\lbrace f_n \rbrace$ obtained by regularisation converges pointwise to $f$ on $(a,b)$. Furthermore, it was also shown that a subsequence of $\lbrace f'_n \rbrace$ converges to $f'$ for a finite number of points in $(a, b)$.
	\end{proof}

	\subsection{Metrizability, Compactness, and Convexity of Positive Operator Monotone functions}
	We recall that $P_1$ consists of operator monotone functions $f: \mathbb{R}^{+} \rightarrow \mathbb{R}^{+}$ such that $f(1) = 1$. Naturally, $P_1$ is a subset of the vector space $E$ of functions $\mathbb{R}^{\mathbb{R}}$. Let $x \in \mathbb{R}$ and consider the pointwise evaluation $\tau_{x}: f \rightarrow f(x)$ for any $f \in E$. Then it is clear that the collection $\lbrace \myabs{\tau_{x}} \rbrace_{x \in \mathbb{R}}$ is a family of separating semi-norms on $E$, so by \cite[Theorem 1.2.6]{kadison-vol-1}, we obtain a locally convex topology on $E$ with base neighbourhoods of the form 
	\begin{equation}\label{locally-convex-neighbourhoods}
		V(f_0: \myabs{\tau_{x_1}} \ldots \myabs{\tau_{x_n}}; \epsilon) = \lbrace f \in E : \myabs{\tau_{x_i}(f) - \tau_{x_i}(f_0)} < \epsilon \, (i = 1, \ldots, n ) \rbrace 
	\end{equation}
	for $f_0 \in E$, $x_1 \ldots x_n \in \mathbb{R}$, $\epsilon > 0$, and $n \in \mathbb{N}$.
	
	It is easy to see that $P_1$ is convex since positive linear combinations of operator monotone functions are again operator monotone. Let us now prove some preliminary properties which will be required to prove that $P_1$ is metrizable and compact \cite[Proposition 28.7]{barry-simon}.
	
	\begin{proposition}\label{barry-simon-operator-monotone-function-properties}
		Let $f: \mathbb{R}^{+} \rightarrow \mathbb{R}^{+}$ be an operator monotone function, and $s > t > 0$, then
		\begin{enumerate}[label=\textnormal{(\Roman*)}]
			\item $0 < f(t) \leq f(s)\leq  \frac{f(t)s}{t}.$
			\item If $f(1)=1$, then $f(t) \leq t + 1.$
			\item If $f(1)=1$, then $0 \leq f(s) - f(t) \leq (1+t^{-1})( s-t) $.
		\end{enumerate}
	\end{proposition}
	\begin{proof}
		\emph{\textnormal{(I)}}.  Since $f$ is bounded below by $0$ and is scalar monotone, then $\displaystyle \lim_{x \rightarrow 0+}f(x)=x_0$ where $x_0 \in [0, \infty)$. Since $f$ is operator monotone, then $f$ is  a scalar concave function by Theorem \ref{operator-monotone-implies-operator-concave}, so $f(s)$ lies below the straight line from $(0, x_0)$ through $(t, f(t))$ which implies
		\begin{equation*}
			\begin{split}
				f(s) &\leq \frac{(f(t)- x_0)s}{t} + x_0 \\
				&= \frac{(f(t) - x_0)(s-t)}{t} + f(t) \\
				&\leq \frac{f(t)(s-t)}{t} + f(t) = \frac{f(t)s}{t}.
			\end{split}
		\end{equation*}
		\par 
		\emph{\textnormal{(II)}}. If $t > 1$, then $f(t) \leq t$ by (I). On the other hand, if $t \leq 1$, then $f(t) \leq f(1) = 1$ by monotonicity. Combining the aforementioned, allows us to conclude that $f(t) \leq t +1$ for any $t > 0$.
		\par 
		\emph{\textnormal{(III)}}. By (I) and (II), 
		\begin{equation*}
			\begin{split}
				0 \leq f(s) - f(t) &\leq \frac{f(t)(s-t)}{t} \leq (1 + t^{-1})(s-t).
			\end{split}
		\end{equation*}
	\end{proof}
	
	\begin{theorem}
		Let $E$ denote the vector space of functions $\mathbb{R}^{\mathbb{R}}$. The subset $P_1$ of $E$ of operator monotone functions $f: \mathbb{R}^{+} \rightarrow \mathbb{R}^{+}$ such that $f(1) = 1$ is metrizable.
	\end{theorem}
	\begin{proof}
		Consider the subspace topology on $P_1$ induced by the topology generated by the base of neighbourhoods of the form in  (\ref{locally-convex-neighbourhoods}) on $E$. Let us show that $P_1$ is second countable. Let $\tau_x$ for some $x \in \mathbb{R}^{+}$ denote the map $\tau_{x}: f \rightarrow f(x)$. Let $f \in E$, $\epsilon > 0$ and  $n \in \mathbb{N}$. Suppose that $g \in V(f: \myabs{\tau_{x_1}}, \ldots, \myabs{\tau_{x_n}};  \epsilon) \cap P_1$ for $\lbrace x_1, \ldots, x_n \rbrace \subseteq \mathbb{R}^{+}$. We also consider the following.
		
		\begin{enumerate}[label=\textnormal{(\Roman*)}]
			\item Let $\delta = \epsilon - \max \lbrace \myabs{f(x_i)-g(x_i)}: i \in  \lbrace 1, \ldots, n \rbrace \rbrace$. Clearly, $\delta > 0$.
			\item Let $p_i \in \mathbb{Q}^{+}$ such that $p_i > x_i$ and $(p_i-x_i)(1+x_i^{-1}) < \delta/4 $ for $ i \in \lbrace 1, \ldots, n \rbrace$.
			\item Let $q_i \in \mathbb{Q}^{+}$ such that $\myabs{g(p_i) - q_i} < \delta/4$ for $ i \in \lbrace 1, \ldots, n \rbrace$.
			\item Let $c : \mathbb{R} \rightarrow \mathbb{R}$ such that $c(p_i) = q_i$ for $ i \in \lbrace 1, \ldots, n \rbrace$.
		\end{enumerate}
		\par
		It is clear that   $ g \in V(c:\myabs{\tau_{p_1}}, \ldots, \myabs{\tau_{p_n}};  \delta /4) \cap P_1$. Let $h \in V(c: \myabs{\tau_{p_1}}, \ldots, \myabs{\tau_{p_n}};  \delta /4) \cap P_1$ and $k \in \lbrace 1, \ldots, n \rbrace$. By (III) of Proposition \ref{barry-simon-operator-monotone-function-properties}, $\myabs{h(x_k) - h(p_k)} < \delta/4$ and $\myabs{g(x_k) - g(p_k)} < \delta /4$. Therefore, 
		\begin{align*}
			\myabs{h(x_k)-g(x_k)} &\leq \myabs{h(x_k)-h(p_k)} + \myabs{h(p_k)-q_k} \\ &+\myabs{q_k - g(p_k)} + \myabs{g(p_k)-g(x_k)}   \\
			&< \delta.
		\end{align*}	   	
		By the above, 
		\begin{equation*}
			\begin{split}
				\myabs{h(x_k)  - f(x_k)} &\leq \myabs{h(x_k) - g(x_k)} + \myabs{g(x_k) - f(x_k)}   \\
				&<  \delta + \myabs{g(x_k) - f(x_k)} \\
				&\leq (\epsilon - \myabs{g(x_k) - f(x_k)}) + \myabs{g(x_k) - f(x_k)}  = \epsilon.
			\end{split}
		\end{equation*}
		Since the above holds for any $k \in \lbrace 1, \ldots, n \rbrace$, we have $h \in V(f: \myabs{\tau_{x_1}}, \ldots, \myabs{\tau_{x_n}} \\;  \epsilon) \cap P_1$, so $g \in V(c: \myabs{\tau_{p_1}}, \ldots, \myabs{\tau_{p_n}}; r) \cap P_1 \subseteq V(f: \myabs{\tau_{x_1}}, \ldots, \myabs{\tau_{x_n}};  \epsilon) \cap P_1$ for some $r \in \mathbb{Q}^{+}$ such that $r < \delta/4$. Hence the space is second countable. Let us show $P_1$ is regular. Suppose $C \subset P_1$ is closed. Let $g \in P_1$ such that $g \notin C$. Therefore, there is some $x \in \mathbb{R}^{+}$ such that $\myabs{g(x)-f(x)} \geq \epsilon$ for $f \in C$. Let $X =  \cup_{f \in C} V(f : \myabs{\tau_{x}}; \epsilon/2)$. Then $C \subseteq X$ and $X \cap V(g: \myabs{\tau_x}; \epsilon/2) = \emptyset$, so $P_1$ is regular.
	\end{proof}

	\begin{lemma}
		The set $P_1$ of operator monotone functions $f: \mathbb{R}^{+} \rightarrow \mathbb{R}^{+}$ such that $f(1) = 1$ is compact.
	\end{lemma}
	\begin{proof}
		Since $P_1$ is metrizable, consider a sequence $\lbrace f_n \rbrace$ in $P_1$ such that $f_n \rightarrow f$ converges pointwise. If $f_n \rightarrow f$ pointwise, it is clear that $f(1)=1$. Since $f(t) \leq t + 1$ by (II) of Proposition \ref{barry-simon-operator-monotone-function-properties}, then by the Borel function calculus and Lebesgue's Dominated Convergence theorem, we can conclude that $f$ is operator monotone. Therefore, $P_1$ is compact. 
	\end{proof}

	\begin{corollary}\label{operator-monotone-functions-can-use-choquet-theory}
		The set $P_1$ of operator monotone functions $f: \mathbb{R}^{+} \rightarrow \mathbb{R}^{+}$ such that $f(1) = 1$ is metrizable, compact and convex in the subspace topology on $P_1$ induced by the topology of $\mathbb{R}^{\mathbb{R}}$. The base of neighbourhoods for the topology on $\mathbb{R}^{\mathbb{R}}$ is given by sets of the form
		\begin{equation*}
			V(f_0: \myabs{\tau_{x_1}} \ldots \myabs{\tau_{x_n}}; \epsilon) = \lbrace f \in \mathbb{R}^{\mathbb{R}} : \myabs{\tau_{x_i}(f) - \tau_{x_i}(f_0)} < \epsilon \, (i = 1, \ldots, n ) \rbrace 
		\end{equation*}
		for $f_0 \in \mathbb{R}^{\mathbb{R}}$, $x_1 \ldots x_n \in \mathbb{R}$, $\epsilon > 0$, and $n \in \mathbb{N}$.
	\end{corollary}
	
	\subsection{Proof}
	Let us first start by characterizing the extreme points of $P_1$. This requires the following results on operator monotone functions \cite[Theorem 10.10]{barry-simon} and \cite[Corollary 4.2]{hansen-fasttrack}.
	\begin{proposition}\label{several-operator-monotone-functions}
		Let $f: \mathbb{R}^{+} \rightarrow \mathbb{R}^{+}$ be an operator monotone function. Then
		\begin{enumerate}[label=\textnormal{(\Roman*)}]
			\item The function $t \rightarrow -t^{-1}f(t)$ is operator monotone on $\mathbb{R}^{+}$.
			\item The function $t \rightarrow t f(t)^{-1}$ is operator monotone on $\mathbb{R}^{+}$.
			\item The function $t \rightarrow t f(t^{-1})$ is operator monotone on $\mathbb{R}^{+}$.
		\end{enumerate}
	\end{proposition}
	\begin{proof}
		For notational convenience, let $Sf(t) = t^{-1}f(t)$  on $\mathbb{R}^{+}$ for any operator monotone function $f$.
		\par 
		\emph{\textnormal{(I)}}. Let $ (a, b) \subsetneq \mathbb{R}^{+}$ and suppose that $\displaystyle \lim_{x \rightarrow 0+}f(x) = f(0+)=  0$.  Consider  $N$ sufficiently large such that $1/N < a$. By Theorem \ref{divided-difference-is-operator-monotone-decreasing}, we can consider the sequence of operator monotone functions $\lbrace -L_{1/n}f \rbrace_{n \geq N}$. Furthermore,
		\[  \lim_{n \rightarrow \infty} L_{1/n}(f)(t) = \lim_{n \rightarrow \infty}  \frac{f(t) - f(n^{-1})}{t - n^{-1}} = t^{-1} f(t), \quad t \in (a, b). \]
		Let $M > 0$. For sufficiently large $n$, we have that $\myabs{L_{1/n}f(t)} \leq M$  for all $t \in (a, b)$, so by using Lebesgue's dominated convergence theorem and the Borel function calculus, $-Sf$ is operator monotone on $(a,b)$ for any $ (a, b) \subsetneq \mathbb{R}^{+}$. Consequently, $-Sf$ is operator monotone on $\mathbb{R}^{+}$.
		\par
		If $f(0+)> 0$, let $g(t) = f(t) - f(0+)$ for all $t \in (0, \infty)$. Therefore, $t^{-1}f(t) = t^{-1}g(t) + t^{-1}f(0+)$. Since $g$ is operator monotone on $\mathbb{R}^{+}$ and $g(0+)=0$, $-Sg$ is operator monotone on $\mathbb{R}^{+}$. Moreover, since $f(0+) > 0$, the function $ t \rightarrow -t^{-1}f(0+)$ is operator monotone on $\mathbb{R}^{+}$. Hence, $-Sf$ is operator monotone on $\mathbb{R}^{+}$.  \\
		\emph{\textnormal{(II)}}. By (I), $-Sf$ is operator monotone on $\mathbb{R}^{+}$, so the function $t \rightarrow (tf(t))^{-1} = t^{-1} f(t)$ is operator monotone on $\mathbb{R}^{+}$. \\
		\emph{\textnormal{(III)}}. By (I), $-Sf$ is operator monotone on $\mathbb{R}^{+}$, so the function $t \rightarrow t^{-1}f(t) \circ t^{-1} = t f(t^{-1})$ is operator monotone on $\mathbb{R}^{+}$.
	\end{proof}
	
	In addition, we shall also need \cite[Corollary 4.7]{hansen-fasttrack}. 
	\begin{proposition}
		Let $f: \mathbb{R}^{+} \rightarrow \mathbb{R}^{+}$ be an operator monotone function. Then $0 \leq f'(1) \leq 1$.
	\end{proposition}
	\begin{proof}
		\emph{\textnormal{(I)}}. Since $f$ is non-decreasing, then $f'(t) \geq 0$ for all $t \in \mathbb{R}^{+}$. By (II) of Proposition \ref{several-operator-monotone-functions}, $g(t) = tf(t)^{-1}$ is operator monotone on $\mathbb{R}^{+}$. Furthermore, 
		\[ g'(t) = (tf(t)^{-1})' = \frac{f(t)-tf^{'}(t)}{(f^{'}(t))^2}. \] 
		Therefore, 
		\begin{equation}\label{derivative-one-formula}
			0 \leq f'(1) \leq f'(1) + g'(1) = 1.
		\end{equation}
	\end{proof}
	
	\begin{lemma}\label{derivatie-one-or-zero-are-extreme-points}
		Let $P_1$ be the set of operator monotone functions $f: \mathbb{R}^{+} \rightarrow \mathbb{R}^{+}$ such that $f(1) = 1$. If $f'(1) = 0$ or $f'(1) = 1$, then $f$ is an extreme point of $P_1$. Furthermore, 
		\begin{enumerate}[label=\textnormal{(\Roman*)}]
			\item If $f'(1) = 0$, then $f(t) =1$ for $t \in \mathbb{R}^{+}$.
			\item If $f'(1) = 1$, then $f(t) =t$ for $t \in \mathbb{R}^{+}$.
		\end{enumerate}
	\end{lemma}
	\begin{proof}
		Let $t_1, t_2 \in (a, b) \subsetneq \mathbb{R}^{+}$. Suppose $\lbrace f_n \rbrace$ is the sequence of infinitely differentiable operator monotone functions generated by Corollary \ref{regularisation-gives-infinitely-differentiable-functions-converging-pointwise} on $(a, b)$. This implies that 
		\[  \lim_{n \rightarrow \infty}  [\Delta f_n(t_i, t_j)]_{i,j=1}^{2}  = [\Delta f(t_i, t_j)]_{i,j=1}^{2} \]
		in s.o.t.  By Theorem \ref{first-divided-difference-positive-implies-matrix-monotone}, $[\Delta f_n(t_i, t_j)]_{i,j=1}^{2} $  is positive semi-definite for all $n \in \mathbb{N}$, so it follows that $[\Delta f(t_i, t_j)]_{i,j=1}^{2} $ is positive semi-definite. Consequently, the determinant is non-negative, so
		\[ \bigg|   \frac{f(t_1) - f(t_2)}{t_1 - t_2} \bigg|^{2} \leq f'(t_1) f'(t_2), \quad \forall \, t_1, t_2 \in \mathbb{R}^{+}. \]
		If $f'(1) = 0$, then $f$ is constant by the above inequality, so $f(t) =1$ for all $t \in \mathbb{R}^{+}$. If $f'(1) = 1$, then by (\ref{derivative-one-formula}), $(tf(t)^{-1})'(1)=0$ which implies that $(tf(t)^{-1})=1$ for all  $t \in \mathbb{R}^{+}$, so  $f(t) = t$ for all  $t \in \mathbb{R}^{+}$.
		
		Let $h_1, h_2 \in P_1$ and $\lambda \in (0,1)$ and suppose that $f = \lambda h_1 + (1-\lambda) h_2$. If $f'(1) = 0$ or $f'(1) = 1$, it is clear that $f'(1) = h_1'(1)$. Therefore, $f=h_1$ for both cases.
	\end{proof}
	
	We are now in a position to prove \cite[Lemma 4.8]{hansen-fasttrack}.
	
	\begin{proposition}\label{extreme-points-of-operator-monotone-functions-characterised}
		The extreme points of the set of operator monotone functions $f: \mathbb{R}^{+} \rightarrow \mathbb{R}^{+}$ such that $f(1) = 1$ have the form 
		\begin{equation}\label{extreme-point-equation}
			f(t) = \frac{t}{f'(1) + (1-f'(1))t }, \quad t > 0.
		\end{equation}
	\end{proposition}
	\begin{proof}
		By Lemma \ref{derivatie-one-or-zero-are-extreme-points}, whenever $f'(1) = 0$ or $f'(1) = 1$, $f$ is an extreme point and is of the form outlined by (\ref{extreme-point-equation}).
		\par 
		Let $0 < f'(1) < 1$. Let $g: t \rightarrow t f(t^{-1})$. By (III) of Proposition \ref{several-operator-monotone-functions}, $g$ is operator monotone. Furthermore, define the following functions.
		\begin{align*} 
			f_1(t) &= \frac{t}{f'(1)}L_1f(t), \\
			f_2(t) &= \frac{1}{1- f'(1)}L_1g(t^{-1}).
		\end{align*}
		
		We show that $f_1$ and $f_2$ are operator monotone. By Theorem \ref{divided-difference-is-operator-monotone-decreasing}, $-L_1f$ is operator monotone  which means that $t \rightarrow L_1f(t^{-1}) $ is operator monotone. By (III) of Proposition \ref{several-operator-monotone-functions}, the function $t \rightarrow tL_1f((t^{-1})^{-1}) = tL_1f(t) $ is operator monotone, so $f_1$ is operator monotone. Similarly, by Theorem \ref{divided-difference-is-operator-monotone-decreasing}, $-L_1g$ is operator monotone  which implies that $t \rightarrow L_1g(t^{-1})$ is operator monotone, so $f_2$ is operator monotone. We now show that $f_1(1)=f_2(1)=1$. Since $L_1f(1) = f'(1)$, then $f_1(1)=1$. Moreover, since $g'(t) = f(t^{-1}) - t^{-1}f'(t^{-1})$, we have $g'(1) = 1  - f'(1)$. This clearly implies that $f_2(1) =1$.
		\par
		Let us now show that
		\begin{equation}\label{extreme-points-of-operator-monotone-functions-characterised-eq1}
			f'(1) f_1(t) + (1-f'(1)) f_2(t) = f(t), \quad \forall \, t \in \mathbb{R}^{+}. 
		\end{equation}
		Since $f_1(1) = f_2(1) = 1$, then the equation clearly holds for when $t=1$. For $t \neq 1$, we compute
		\begin{align*}
			f'(1) f_1(t)  & =  \frac{t(f(t)-1)}{t-1}.    \\
			(1-f'(1)) f_2(t)  & =  \frac{t^{-1}f(t)-1}{t^{-1}-1} = \frac{t -f(t)}{t-1}.   
		\end{align*} 
		Moreover, 
		\[  \frac{t(f(t)-1)}{t-1}  + \frac{t -f(t)}{t-1} = f(t) \]
		which proves (\ref{extreme-points-of-operator-monotone-functions-characterised-eq1}).
		Since $f$ is an extreme point, 
		\[ f(t) = f_1(t) = \frac{t}{f'(1)} \frac{f(t) - 1}{t-1}, \quad t \in \mathbb{R}^{+} \setminus \lbrace1\rbrace. \]
		So, 
		\[ f(t) = \frac{t}{f'(1) + (1-f'(1))t}, \quad t \in \mathbb{R}^{+}.  \]
	\end{proof}
	This leads us to the proof of L\"owner's Theorem.
	\begin{theorem}
		Let $f$ be a positive operator monotone function defined in the positive half-line. There is a  positive Radon measure $\mu$ on the closed interval $[0,1]$ such that 
		\begin{equation*}
			f(t) = \int_{[0,1]} \frac{t}{\lambda + (1-\lambda)t}  \,  {\rm d} \mu(\lambda), \quad t > 0.
		\end{equation*}
		Conversely, any function given on this form is operator monotone. The measure $\mu$ is a probability measure iff $f(1) =1$.
	\end{theorem}
	\begin{proof}
		Suppose $f(1) = 1$. By Theorem \ref{choquet-theorem} and Corollary \ref{operator-monotone-functions-can-use-choquet-theory}, there exists a positive radon measure $\mu$ which represents $f$ on the set of extreme points of the set $P_1$ of operator monotone functions $f: \mathbb{R}^{+} \rightarrow \mathbb{R}^{+}$ such that $f(1) = 1$. Consider the continuous linear functional $\rho$ such that $\rho(f) = f(t)$ where $t \in \mathbb{R}^{+}$ and $f \in P_1$. Then 
		\[ f(t) = \rho(f)  = \int_{\text{Ext}(P_1)} \rho(f_{\lambda}) \, {\rm d} \mu(\lambda).  \]
		Since the extreme points are of the form (\ref{extreme-point-equation}), then
		\[  f(t) =  \int_{[0,1]} \frac{t}{\lambda + (1-\lambda)t} \, {\rm d} \mu(\lambda). \]
	\end{proof}
	
	\bibliographystyle{plain}
	\bibliography{loewner_theorem_review}

\begin{thebibliography}{10}

\bibitem{bendat}
Julius Bendat and Seymour Sherman.
\newblock Monotone and convex operator functions.
\newblock {\em Trans. Amer. Math. Soc.}, 79:58--71, 1955.

\bibitem{conway}
John~B. Conway.
\newblock {\em A course in functional analysis}, volume~96 of {\em Graduate
  Texts in Mathematics}.
\newblock Springer-Verlag, New York, second edition, 1990.

\bibitem{donoghue}
William~F. Donoghue, Jr.
\newblock {\em Monotone matrix functions and analytic continuation}, volume
  Band 207 of {\em Die Grundlehren der mathematischen Wissenschaften}.
\newblock Springer-Verlag, New York-Heidelberg, 1974.

\bibitem{hansen-fasttrack}
Frank Hansen.
\newblock The fast track to {L}\"owner's theorem.
\newblock {\em Linear Algebra Appl.}, 438(11):4557--4571, 2013.

\bibitem{horn-book}
Roger~A. Horn and Charles~R. Johnson.
\newblock {\em Topics in matrix analysis}.
\newblock Cambridge University Press, Cambridge, 1991.

\bibitem{kadison-vol-1}
Richard~V. Kadison and John~R. Ringrose.
\newblock {\em Fundamentals of the theory of operator algebras. {V}ol.~{I}},
  volume 100 of {\em Pure and Applied Mathematics}.
\newblock Academic Press, Inc. [Harcourt Brace Jovanovich, Publishers], New
  York, 1983.
\newblock Elementary theory.

\bibitem{kubo-ando-main}
Fumio Kubo and Tsuyoshi Ando.
\newblock Means of positive linear operators.
\newblock {\em Math. Ann.}, 246(3):205--224, 1979/80.

\bibitem{molnar-main}
Lajos Moln\'{a}r.
\newblock On the order determining property of the norm of a {K}ubo-{A}ndo mean
  in operator algebras.
\newblock {\em Integral Equations Operator Theory}, 93(5):Paper No. 53, 25,
  2021.

\bibitem{niculescu}
Constantin~P. Niculescu and Lars-Erik Persson.
\newblock {\em Convex functions and their applications}, volume~23 of {\em CMS
  Books in Mathematics/Ouvrages de Math\'ematiques de la SMC}.
\newblock Springer, New York, 2006.
\newblock A contemporary approach.

\bibitem{phelps}
Robert~R. Phelps.
\newblock {\em Lectures on {C}hoquet's theorem}, volume 1757 of {\em Lecture
  Notes in Mathematics}.
\newblock Springer-Verlag, Berlin, second edition, 2001.

\bibitem{barry-simon}
Barry Simon.
\newblock {\em Loewner's theorem on monotone matrix functions}, volume 354 of
  {\em Grundlehren der mathematischen Wissenschaften [Fundamental Principles of
  Mathematical Sciences]}.
\newblock Springer, Cham, 2019.

\end{thebibliography}

\end{document}